\documentclass[11pt]{amsart}
\usepackage{amsmath,amssymb,enumerate}
\usepackage{epsf,epsfig,amsfonts,graphicx,color,url}
\usepackage{pgf,tikz}
\usepackage{ulem}

\usetikzlibrary{arrows}
 \numberwithin{equation}{section}

\newtheorem{theorem}{\rm\bf Theorem}[section]
\newtheorem{proposition}[theorem]{\rm\bf Proposition}
\newtheorem{lemma}[theorem]{\rm\bf Lemma}
\newtheorem{corollary}[theorem]{\rm\bf Corollary}
\theoremstyle{definition}
\newtheorem{definition}[theorem]{\rm\bf Definition}

\newtheorem{remark}[theorem]{\rm\bf Remark}

\newcommand{\weg}[1]{}

\begin{document}
\title{Chains in CR geometry as geodesics of a Kropina metric }
\author[J.-H. Cheng]{Jih-Hsin Cheng}
\address{Institute of Mathematics, Academia Sinica, Taipei and National
Center for Theoretical Sciences, Taipei Office, Taiwan, R. O. C.  \\     \url{jhcheng@gate.sinica.edu.tw}}

\author[T. Marugame]{Taiji Marugame}
\address{Institute of Mathematics, Academia Sinica, Taipei, Taiwan, R. O. C. \\  \url{marugame@gate.sinica.edu.tw}
}

\author[V. S. Matveev]{ Vladimir S. Matveev}
\address{Institut f\"ur Mathematik, Friedrich-Schiller Universit\"at Jena 
07737 Jena, Germany,  \   \url{vladimir.matveev@uni-jena.de}}

\author[R. Montgomery]{Richard Montgomery}
\address{University of California, Santa Cruz, Santa Cruz, CA 95064, USA \\   \url{rmont@ucsc.edu}}

\begin{abstract}
With the help of a generalization of the Fermat principle in general
relativity, we show that chains in CR geometry are geodesics of a certain
Kropina metric constructed from the CR structure. We study the projective
equivalence of Kropina metrics and show that if the kernel distributions of
the corresponding 1-forms are non-integrable then two projectively
equivalent metrics are trivially projectively equivalent.   As an application,
we show that sufficiently many chains determine the CR structure up to
conjugacy, generalizing and reproving the main result of \cite{cheng}. The correspondence between geodesics of the Kropina metric and chains allows us to use the methods of metric geometry and the  calculus of variations to study chains.  We use these methods    to re-prove the result of \cite{Jac1}  that locally 
 any two points of a strictly pseudoconvex CR manifolds  can be  joined by a chain. Finally, we  generalize this result to the global setting
 by  showing  that  any two points of a connected compact strictly pseudoconvex CR
manifold which admits a pseudo-Einstein contact form with positive
 Tanaka-Webster scalar curvature can be  joined by a chain.
\end{abstract}

\maketitle



\section{Introduction}

A non-degenerate CR structure on a $2n+1$ dimensional manifold $M$ is a pair $(H,J)$ where $%
H$ is a contact distribution and $J$ is a complex structure on $H$
satisfying a certain integrability condition.  See Section  \ref{CR} for details.

The \textit{chains} for a  CR geometry  are a family of curves on $M$
which are canonically constructed from the non-degenerate CR structure. See, for instance,
the book \cite{Jac} and the references therein. For any point of $M$ and any
direction not contained in $H$ there exists precisely one chain through this
point and tangent to this direction.

One of several equivalent definitions of chains goes through the Fefferman
metric introduced in \cite{fefferman}, which is an indefinite metric on a
circle bundle over $M$. The conformal class of the Fefferman metric is
canonically constructed from the  CR structure on $M$. The infinitesimal
generator $K$ of the circle action is a null  
Killing vector field for this metric. Chains are then defined to be the
projections to $M$ of 
null geodesics for this pseudo-Riemannian metric to $M$. See Definition \ref%
{def:4.1}. 
 
On the other hand, the \textit{Kropina metric} is a function on $TM$ given
by 
\begin{equation*}
F(x,\xi )=\frac{g(\xi ,\xi )}{\omega (\xi )}
\end{equation*}
\noindent where $g$ is a metric and $\omega $ is a nonvanishing 1-form on $M$%
. In this paper we allow metrics $g$ of all signatures; in fact, we will
also allow certain degenerate metrics: as it will become clear below, in order
to define geodesics it is sufficient that the restriction of $g$ to $\mathrm{%
ker}\,\omega$ is non-degenerate.

Kropina metrics are popular objects in Finsler geometry, despite the fact
that they are not strictly speaking Finsler metrics even when $g$ is
positive definite. Indeed $F$ is undefined for $\xi \in \mathrm{ker}\,\omega$%
.

Our first result is the following theorem.

\begin{theorem}
\label{thm:1} Chains are geodesics of a certain Kropina metric.
\end{theorem}

By  a {\it geodesic} of a Kropina metric $F= g/\omega$ we   mean any smooth regular solution  $\gamma(t)$ 
of the Euler-Lagrange equation for the Lagrangian   $F$ satisfying the additional property
 that  $F(\gamma(t), \gamma'(t))$ is defined  (i.e., $ \omega(\gamma'(t))\ne 0$)
for all $t$.  

 In Section 4.3, we give a formula for
this Kropina metric. When $M$ is the boundary of a strictly pseudoconvex domain 
and $\rho$ is Fefferman's defining function for $M$ then we can   express this
 metric as $F=(\partial\overline\partial%
\rho)/{\rm Im}(\partial \rho)$  where we regard $\partial\overline\partial%
\rho$ as a symmetric 2-tensor restricted to $M$ and ${\rm Im}(\partial \rho)$ as a one-form restricted to $M$. 

In the above theorem and throughout this paper,  with the exception of Section \ref{indicatrix},  we consider geodesics
without preferred parameterization.  Clearly, $F(x, \xi)= -F(x, -\xi)$, so that for a geodesic $\gamma(t)$ the curve $\gamma(C -t)$ for an appropriate constant $C$ is   also a geodesic. For any point $(x,\xi)\in  TM$ such that $\omega(\xi)\ne 0$,  there exists a local geodesic of $g/\omega$,  unique up to re-pameterization,  which starts at $x$  tangent to  $\xi$.     

To formulate the next result and apply it to CR
geometry, we only need to know that $F$ has the form $g/\theta$, where $%
\theta $ is a contact form for the underlying contact structure on the CR
manifold $M$. The metric $g$ is defined up to the transformation $g\mapsto
g+\theta \cdot \beta $ with a closed 1-form $\beta $. Clearly this
transformation corresponds to an addition of the closed form $\beta $ to the
Kropina metric, and does not change its geodesics.

 The Kropina metric also has a relation to a different topic in CR geometry: 
We consider the energy functional $E(\gamma)$, the integral of the Kropina metric over 
a curve $\gamma $ transversal to the contact distribution in the CR manifold $M$. 
Suppose $M$ bounds an asymptotic complex hyperbolic domain $\Omega$ and $\gamma $ 
bounds a minimal surface $\Sigma $ in $\Omega $. In \cite{CMYZ}, $E(\gamma )$ was shown
to appear as the log term coefficient in the area renormalization expansion of $\Sigma $ for $\dim M=3.$

The next result concerns the  projective equivalence of Kropina metrics. 
According to the classical definition,  two geometric structures (Riemannian, Finsler,
or affine connections) are projectively equivalent  if they have the same
geodesics. In the case of Kropina metrics, we will  modify  this
definition.  The reason for this modification is that Kropina
geodesics are not defined precisely  along  directions  lying in the kernel of $%
\omega $. Thus if two Kropina metrics are projectively equivalent according to  the
classical definition, then their 1-forms coincide up to scale, which is an extremely
strong additional condition. Note that Theorem \ref{connection} below shows that for certain Kropina metrics we  can not reconstruct the kernel of $\omega$  by the geodesic equation. 
\weg{On the other hand, in certain cases one can
also naturally define geodesics of a Kropina metric in the direction of
vectors lying in the kernel of $\omega $ with a kind of limit procedure.
Actually we will see that for certain Kropina metrics there is no difference
on the level of equations for geodesics between vectors lying and not lying
in the kernel of $\omega $. Our definition of projective equivalence will
allow to include these \textquotedblleft additional\textquotedblright\
geodesics into consideration.}

We call a set of curves on the manifold $M$ \textit{sufficiently big} if for
any point $p\in M$ the set of tangent vectors at $p$ for these curves
contains a nonempty open subset of $T_{p}M$. We call two Kropina metrics 
\textit{projectively equivalent}, if there exists a sufficiently big set of
curves which are geodesics for both metrics.

Our second result is the following theorem:

\begin{theorem}
\label{thm:2} Suppose two Kropina metrics $F=g/\omega$, $\widehat F=\widehat
g/\widehat \omega$ are projectively equivalent and satisfy the following
conditions:

\begin{itemize}
\item $\mathrm{ker}\,\omega$ is non-integrable   i.e., $%
\omega\wedge d\omega \ne 0$ at almost every point; 

\item $g,\widehat{g}$ are non-degenerate on $\mathrm{ker}\,\omega $, $%
\mathrm{ker}\,\widehat{\omega }$ respectively.
\end{itemize}

Then, $\omega = \alpha\widehat\omega$ for a certain non-vanishing function $%
\alpha$ and there exist a constant $c\ne 0$ and a closed 1-form $\beta$ such
that $\widehat F = c F + \beta$.
\end{theorem}

Combining Theorem \ref{thm:1} and Theorem \ref{thm:2}, we obtain the
following generalization of the results of \cite{cheng, CZ}:

\begin{corollary}
\label{chain-diffeo} Suppose two non-degenerate CR structures $(H,J)$ and $(%
\widehat{H},\widehat{J})$ have the same sufficiently big family of chains.
Then these two CR structures coincide or are conjugate: that is, $H=\widehat{%
H}$ and either $J=\widehat{J}$ or $J=-\widehat{J}$.
\end{corollary}

Note that  \cite{cheng} assumes  that \textit{all} chains of one
structure are chains of the other structure. 
As explained
above,  this implies that the corresponding 1-forms are proportional. The latter assumption essentially simplifies the proof. 
 Similarly,   \cite{CZ}   requires that the corresponding contact distributions coincide.  

Theorem \ref{thm:2} describes all pairs of projectively equivalent Kropina
metrics such that for at least one of   them  the kernel
distribution of the corresponding 1-form is non-integrable. Let us now
consider the remaining case, i.e., when for both 1-forms the kernel
distributions are integrable. In this case without loss of generality we can
assume that the 1-forms are closed.  The next theorem shows that then the
geodesics are geodesics of a certain affine connection.

\begin{theorem}
\label{connection} If the 1-form $\omega$ is closed, then for any Kropina
metric $F=g/\omega$ there exists an affine connection $\nabla=
(\Gamma_{ij}{}^k)$ such that each geodesic of $F$ is a geodesic of $\nabla$.
\end{theorem}

The precise formula for the connection is in Theorem \ref{thm:con}. It is torsion-free.

In particular, in dimension 2, all Kropina metrics are projectively
equivalent to affine connections. This is actually known and was one of the
motivations for introducing Kropina metrics, see \cite{Kropina}.

Note that the question when two affine connections are projectively
equivalent is well-understood, see e.g. \cite{Matveev2012}.

  Chern and Moser \cite[p. 222]{ChernMoser} told us to think of chains as the CR
versions of geodesics.  If their analogy is a good one, then any two
sufficiently nearby points ought to be connected by a chain, and if $M$ were
compact and connected, then any two points at all ought to be connected by a
chain. The first assertion does hold for strictly pseudoconvex CR manifolds:
nearby points are connected by chains. See \cite{Jac}, p. 185, and the
original references therein, including \cite{Jac1,Koch}. Surprisingly, the
second assertion is false, even if the compact manifold is locally CR
equivalent to the standard model, $S^{3}$ with its canonical strictly
pseudoconvex structure. This example, the Burns-Shnider counterexample, is
detailed in \cite[p. 185]{Jac}. See also the original reference, \cite{BS}.
 
The fact that   chains are geodesics of a Kropina metric allows us  to employ variational methods and  techniques of metric geometry to investigate chains. 
We will use these methods to 
  reprove and generalize   the famous result of  \cite{Jac1, Koch} 
 on local chain connectivity.

\begin{theorem}
\label{connect-thm} Let $F = g/\omega$ be a Kropina metric on $M$ with $g$
positive definite.  Then, the following statements hold: 

\begin{itemize}
\item[(A)] If at $p\in M$ we have
 $\omega
\wedge d \omega \ne 0$, then there exists a  neighborhood $U$ of $p$  such
that for  any  $q \in U$ one  can  join $p$ to $q$ by a length minimizing
Kropina geodesic which does not leave $U$.  This neighborhood $U$ can be chosen to be arbitrary small.

\item[(B)] Suppose that  $M$ is compact   and assume that the set of the points $p\in M$ 
such that $\omega
\wedge d \omega \ne 0$ is  connected and  everywhere dense in $M$. 
Then,  any two points of $M$ can be
joined by a length-minimizing Kropina geodesic $\gamma$  such that  $F(\gamma(t), \gamma'(t))>0 $ at each point. 
\end{itemize}
\end{theorem}

We will see that for  strictly pseudoconvex
CR manifolds the corresponding metric $g$ can be made positive definite locally. This follows
from Lemma \ref{modify} applied to the Kropina metric associated to a strictly pseudoconvex
CR manifolds. Moreover, on a strictly pseudoconvex CR manifold admitting  a pseudo-Einstein contact form with positive
Tanaka-Webster scalar curvature, the metric $g$ can be made positive definite globally (see (\ref{global-kropina})). On
strictly pseudoconvex CR manifolds we have by definition  that  $\omega
\wedge d \omega \ne 0$  at all points. Therefore, we obtain:

\begin{corollary}
\label{cr-connect-thm} Let $M$ be a strictly pseudoconvex CR manifold. Then the following statements hold: 

\begin{itemize}
  \item[(A)] Each point $p\in M$ has a   neighborhood $U$ such that   any $q \in U$ is connected to  $p$  by 
a chain lying  in $U$. 

\item[(B)] If $M$ is connected,  compact and admits a pseudo-Einstein contact form with positive
Tanaka-Webster scalar curvature, then, any two points of $M$ can be joined by a chain.
\end{itemize}
\end{corollary}
 
 The Burns-Shnider   example    mentioned above is a CR structure on a compact manifold $S^{2n}\times S^{1}$ such that not every two points can be connected by a chain. 
It follows from Corollary \ref{cr-connect-thm} that we cannot find a
pseudo-Einstein contact form with positive scalar curvature for this CR structure. We discuss this fact in  detail in  Section \ref{BS}.  

Our paper is organised as follows. In Section \ref{sec:fermat}, we recall the
Fermat principle from general relativity and generalize it to the case when
the Killing vector field is null; this will give us the proof of Theorem \ref%
{thm:1}. In Section  \ref{sec:kropina}, we examine some properties of the set of
2-jets of solutions to the Euler-Lagrange equation for the Kropina metric
and prove Theorem \ref{thm:2} and Theorem \ref{connection}. In Section  \ref%
{chain-Kropina}, we describe the Kropina metric associated to the Fefferman
metric, and apply our results to prove Theorem \ref{thm:1} and Corollary \ref%
{chain-diffeo}. We also analyse  the Burns-Shnider example  in detail and show where the positivity of the scalar curvature fails.  In Section  \ref{indicatrix}, we describe 
  the indicatrix of the Kropina metric   and   prove
Theorem \ref{connect-thm}.    All objects in our papers are assumed to be sufficiently smooth; $C^2$-smoothness is enough for all proofs  related to Kropina metrics.   \bigskip 

\textbf{Acknowledgements: } The main results were obtained during the visit
of V.M. to Taipei supported by Academia Sinica and DFG. The topic was
suggested by R.M. in the open problem session of the workshop \textit{%
Analysis and Geometry on pseudohermitian manifolds} held at the American
Institute of Mathematics (AIM) and organized by Sorin Dragomir, Howard
Jacobowitz, and Paul Yang in November 2017. AIM supported the participation
of J.-H. C., V.M. and R.M. at this workshop. R.M. would like to thank the
NSF, grant DMS-20030177 and email discussions with Andrei Agrachev and Mikhail Zhitomirskii. 
 V.M. thanks Sergei Ivanov for a hint  which helped   prove Theorem \ref{connect-thm}.
J.-H. C. would like to thank the Ministry of Science and Technology of
Taiwan, R.O.C. for the support of the project: MOST 106 - 2115 - M - 001 -
013- and the National Center for Theoretical Sciences for the constant
support.

\bigskip

\section{Fermat Principle and Proof of Theorem \protect\ref{thm:1}}

\label{sec:fermat}

We start by recalling the Fermat Principle. For simplicity we restrict to a
local version.

Consider a Lorentzian metric $\widetilde{g}$ on $\widetilde{M}$ admitting a
time-like Killing vector field $K$ and a local space-like hypersurface $M$.
Consider the (Randers) Finsler metric on $M$ given by the formula 
\begin{equation*}
F_R(x,\xi )=\sqrt{\widetilde{g}(\xi ,\xi )+\widetilde{g}(K,\xi )^{2}}+%
\widetilde{g}(K,\xi )
\end{equation*}%
\noindent for $\xi \in T_{x}M$. The projection of a null geodesic $\gamma(t)$
to $M$ is the curve $\phi (\tau (t))\circ \gamma (t)$ where $\phi $ denotes
the flow of $K$ and the `time function' $\tau (t)$ is such that for each $t$
the point $\phi (\tau (t))\circ \gamma (t)$ lies on $M$. If we work locally,
in a small neighborhood of a point of $M$, no ambiguity appears. Then we
have:

\bigskip

\textbf{Fermat Principle} (see e.g. \cite{sanchez,perlick}): The projected
null geodesics $\phi (\tau (t))\circ \gamma (t)$ are geodesics of the metric 
$F_R$. Conversely, for any point $p\in M$ and for any geodesic $\gamma$ of $F_R$ 
starting from $p$, there exists a null geodesic $\widetilde\gamma$ of $\widetilde g$ 
such that $\widetilde g(K, \dot{\widetilde\gamma})<0$ starting from $p$ whose projection 
is $\gamma$. This geodesic is unique up to orientation-preserving 
reparameterisations.

\bigskip

The condition  above that the geodesic is oriented is important: indeed,  if the one-form $\xi\mapsto \widetilde g(K,\xi)$ is not closed, the metric $F_R$  is not geodesically reversible and 
for most geodesics $\gamma$ of $F_R$ the  same curve but with the reversed orientation is not a geodesic of $F_R$, see e.g. \cite{M2012}. Actually,  geodesics with reversed orientation, so called {\it backward geodesics}, correspond to projections of   null geodesics with $\widetilde g(K, \dot{\widetilde\gamma})>0$ to $M$.

Let us now generalize this statement to the case when the Killing vector
field is null.

Let $(\widetilde{M},\widetilde{g})$ be an $(n+1)$-dimensional Lorentzian or
pseudo-Riemannian manifold with a nonvanishing null Killing vector field $K$%
. Let $M$ be a hypersurface transverse to $K$. Using the flow generated by $K
$, we form a local coordinate system $(x^{0},x^{1},\dots ,x^{n})$ around a
point $p\in M$ such that $M=\{x^{0}=0\}$ and $K=\partial /\partial x^{0}$.
Then, $\widetilde{g}$ is written in the form 
\begin{align*}
\widetilde{g}& =g_{ij}dx^{i}dx^{j}+2\omega _{i}dx^{i}dx^{0} \\
& =%
\begin{pmatrix}
0 & \omega _{j} \\ 
\omega _{i} & g_{ij}%
\end{pmatrix}%
,
\end{align*}%
where the components $g_{ij}$, $\omega _{i}$ are functions of $(x^{1},\dots
,x^{n})$ and are independent of $x^{0}$ since $\widetilde{g}$ is invariant
under the flow of $K$ (Hereafter, Latin indices $i,j,k,\dots $ run from $%
1$ to $n$ and we adopt Einstein's summation convention.) We  view $%
\omega :=\omega _{i}dx^{i}$ as a 1-form on $M$ and $g$ as a (perhaps
degenerate) metric on $M$. Clearly $\omega =(K\lrcorner \widetilde{g})|_{TM}$%
. The \textit{Kropina metric} $F$ on $M$ associated to $\widetilde{g}$ is
defined by 
\begin{equation}
F(x,\xi ):=\frac{\widetilde{g}(\xi ,\xi )}{\widetilde{g}(K,\xi )}=\frac{%
g_{ij}\xi ^{i}\xi ^{j}}{\omega _{l}\xi ^{l}}  \label{Kropina-Lorentz}
\end{equation}%
for $\xi \in T_{x}M\setminus \mathrm{ker}\,\omega $. Note that $F$ depends
only on the conformal class of $\widetilde{g}$. If we consider another
hypersurface $M^{\prime }$ defined by $x^{0}=f(x^{1},\dots ,x^{n})$ and
identify it with $M$ by the flow of $K$, then the Kropina metric changes by
adding of the exact form $2df$ and so has the same geodesics.

Now let $\pi (x^{0},x^{i})=(x^{i})$ be the (local) projection $\widetilde{M}%
\rightarrow M$ along the integral curves of $K$. Then the correspondence
between 
  null geodesics of $\widetilde{g}$ and
  geodesics of $F$ is as follows:

\begin{theorem} [cf. {\protect\cite[Theorem 7.8]{Caponio}}] \label{fermat} 
If $\widetilde{\gamma}(t)$ is a null geodesic of $\widetilde
g $ with $\widetilde{g}(K, \dot{\widetilde{\gamma}}(0))\neq 0$, then $%
\gamma(t):=\pi(\widetilde{\gamma}(t))$ is a geodesic of $F$. Conversely, if $%
\gamma(t)$ is a geodesic for $F$ then  there exists a null geodesic $\widetilde{%
\gamma}(t)$ for $\widetilde g$ with $\pi(\widetilde{\gamma}(t))=\gamma(t)$  and  which is    uniquely determined 
 by the choice of $%
\widetilde{\gamma}(0)\in\pi^{-1}(\gamma(0))$.
\end{theorem}
In other words, as in the Fermat Principle, null geodesics of $\widetilde g$ are essentially the same as  geodesics of the Kropina metric  $F$: locally, given a null geodesic starting at $p\in M\subset\widetilde M$ its projection to $M$ is a geodesic of $F$, and given a geodesic of $F$ starting at  $p\in M$ there locally  exists an unique  null  geodesic of $\widetilde g$ starting at $p$ whose projection  is the geodesic of $F$.     
\begin{proof}
We work in the local coordinate system defined above. Suppose $\widetilde{
\gamma }(t)=(x^{0}(t),x^{i}(t))$ is a null geodesic of $\widetilde{g}$ with $%
\widetilde{g}(K,\dot{\widetilde{\gamma }}(0))=\omega _{l}\dot{x}^{l}(0)\neq
0 $. Let 
\begin{equation*}
L:=g_{ij}\xi ^{i}\xi ^{j}+2\omega _{l}\xi ^{l}\xi ^{0}
\end{equation*}%
be the Lagrangian for the geodesic of $\widetilde{g}$. The Euler-Lagrange
equation 
\begin{equation*}
\frac{d}{dt}\Bigl(\frac{\partial L}{\partial \xi ^{0}}\Bigr)-\frac{\partial L%
}{\partial x^{0}}=0
\end{equation*}%
implies that $\omega _{l}\dot{x}^{l}$ is constant, so we set $a:=(\omega _{l}%
\dot{x}^{l})^{-1}\in \mathbb{R}$. The nullity condition gives 
\begin{equation*}
2\dot{x}^{0}=-ag_{ij}\dot{x}^{i}\dot{x}^{j}.
\end{equation*}%
Then we compute 
\begin{align*}
\frac{\partial L}{\partial x^{k}}& =(\partial _{k}g_{ij})\dot{x}^{i}\dot{x}%
^{j}+2(\partial _{k}\omega _{l})\dot{x}^{l}\dot{x}^{0}=(\partial _{k}g_{ij})%
\dot{x}^{i}\dot{x}^{j}-a(\partial _{k}\omega _{l})g_{ij}\dot{x}^{l}\dot{x}%
^{i}\dot{x}^{j}, \\
\frac{\partial L}{\partial \xi ^{k}}& =2g_{kj}\dot{x}^{j}+2\omega _{k}\dot{x}%
^{0}=2g_{kj}\dot{x}^{j}-a\omega _{k}g_{ij}\dot{x}^{i}\dot{x}^{j}.
\end{align*}%
On the other hand, we have 
\begin{align*}
\frac{\partial F}{\partial x^{k}}& =(\omega _{l}\dot{x}^{l})^{-1}(\partial
_{k}g_{ij})\dot{x}^{i}\dot{x}^{j}-(\omega _{l}\dot{x}^{l})^{-2}(\partial
_{k}\omega _{l})\dot{x}^{l}g_{ij}\dot{x}^{i}\dot{x}^{j} \\
& =a(\partial _{k}g_{ij})\dot{x}^{i}\dot{x}^{j}-a^{2}(\partial _{k}\omega
_{l})g_{ij}\dot{x}^{l}\dot{x}^{i}\dot{x}^{j}, \\
\frac{\partial F}{\partial \xi ^{k}}& =2(\omega _{l}\dot{x}^{l})^{-1}g_{kj}%
\dot{x}^{j}-(\omega _{l}\dot{x}^{l})^{-2}\omega _{k}g_{ij}\dot{x}^{i}\dot{x}%
^{j} \\
& =2ag_{kj}\dot{x}^{j}-a^{2}\omega _{k}g_{ij}\dot{x}^{i}\dot{x}^{j}.
\end{align*}%
Hence we obtain 
\begin{equation*}
\frac{\partial F}{\partial x^{k}}=a\frac{\partial L}{\partial x^{k}},\quad 
\frac{\partial F}{\partial \xi ^{k}}=a\frac{\partial L}{\partial \xi ^{k}}.
\end{equation*}%
Since $a$ is constant, the Euler-Lagrange equation for $L$ implies that for $%
F$.

Conversely, suppose $\gamma (t)=(x^{i}(t))$ is a geodesic of $F$. By
definition, we have $\omega _{l}\dot{x}^{l}(0)\neq 0$ and by changing the
parametrization we may assume that $\omega _{l}\dot{x}^{l}$ is constant. For
any point $\widetilde{y}\in \pi ^{-1}(\gamma (t))$, there exists a unique
null tangent vector at $\widetilde{y}$ which projects to $\dot{\gamma}(t)$,
and this assignment defines a null vector field along the fiber $\pi
^{-1}(\gamma )$. Then, for each choice of $\widetilde{\gamma }(0)\in \pi
^{-1}(\gamma (0))$ the integral curve of this vector field projects to $%
\gamma $, and by the above computations, satisfies the Euler-Lagrange
equation for $L$. Thus we have completed the proof.
\end{proof}

We can also perform the inverse construction: Suppose that $F(x,\xi
)=g_{ij}\xi ^{i}\xi ^{j}/\omega _{l}\xi ^{l}$ is a Kropina metric on $M$
such that $g$ is non-degenerate on $\mathrm{ker}\,\omega $. Then $\widetilde{%
g}:=g_{ij}dx^{i}dx^{j}+2\omega _{i}dx^{i}dx^{0}$ defines a pseudo-Riemannian
metric on $M\times \mathbb{R}$ with $\partial /\partial x^{0}$ being a null
Killing vector field, and the associated Kropina metric is given by $F$.
Thus, by Theorem \ref{fermat}, we have the following

\begin{corollary}
\label{geodesic} Let $F(x,\xi )=g_{ij}\xi ^{i}\xi ^{j}/\omega _{l}\xi ^{l}$
be a Kropina metric on $M$ such that $g$ is non-degenerate on $\mathrm{ker}%
\,\omega $. For any $p\in M$ and $\xi \in T_{p}M{\setminus \mathrm{ker}%
\,\omega}$, there exists a unique geodesic $\gamma (t)$ of $F$ with $\gamma
(0)=p$, $\dot{\gamma}(0)=\xi $.
\end{corollary}

\textsc{Proof of Theorem \ref{thm:1}.} Apply Theorem \ref{fermat} to the
Fefferman metric (see Section  \ref{chain-Kropina}) to obtain that chains are
geodesics of the associated Kropina metric. From the construction of the
Fefferman metric described in Section  \ref{chain-Kropina} we also see that the
one-form $\omega $ forming the denominator of the Kropina  metric is a
contact form for the contact distribution $H$. \qed 

\section{Projective equivalence of Kropina metrics and proof of  
Theorems \protect\ref{thm:2} and \protect\ref{connection}} 

\label{sec:kropina}

Let $F(x,\xi )=g(\xi ,\xi )/\omega (\xi )$ be a Kropina metric on $M$.
For most of this section we assume that

\begin{itemize}
\item $H:= \mathrm{ker}\,\omega$ is non-integrable at every point;

\item $g$ is non-degenerate on $H $.
\end{itemize}

Our   first  goal is to show that a sufficiently big family of
geodesics determines $F $ up to  transformations of the form $\widehat{F}%
=cF+\beta $, where $c$ is a constant and $\beta $ is a closed 1-form. We
begin by reducing to the case where $g$ is a pseudo-Riemannian metric.

\begin{lemma}
\label{modify} In a neighborhood of any point $p\in M$, there exists a
smooth function $f$ such that $g+\omega\cdot df$ is a pseudo-Riemannian metric. If $g$ is positive definite on $H$, then   $g+\omega\cdot df$ is 
also  positive definite for a certain $f$. 
\end{lemma}

\begin{proof}
Since $g|_{H}$ is non-degenerate, we have the orthogonal decomposition $%
T_{p}M=\mathbb{R}X\oplus H_p$, where $X$ is a basis of the one-dimensional
kernel of $g:T_{p}M\rightarrow H_p ^{\ast}$. If we choose a function $f$ so
that $\mathrm{ker}\,(df)_{p}=H_p$ and $g(X,X)+\omega (X)df_p (X)>0$, then $%
g+\omega \cdot df$ is non-degenerate near $p$. In particular, when $g|_H$ is positive definite, $g+\omega\cdot df$ is a Riemannian metric. 
\end{proof}

Next, we give an algebraic description of the space of the 1st and the 2nd
derivatives of solutions to the unparameterized geodesic equations --- i.e.,
solutions to the Euler-Lagrange equations for $F$. To this end, let $J^2 M$
denote the space of 2nd jets of curves on $M$. $J^2 M $ forms a fiber bundle
over $M$ with fiber $\mathbb{R}^{2n}$. 
A local coordinate system $(x^{1},\dots ,x^{n})$ around a point $p\in M$
induces fiber coordinates $\xi^i, \eta^j$, $i, j =1, \ldots ,n$ on $J^2 M$
with  $\xi^i$ standing for $\dot x^i $ and $\eta^j$ for $\ddot x^j$ with $%
x^i (t)$ being a curve germ in $M$.  There is a canonical  projection $J^2 M \to TM$ which
just keeps the first order Taylor information,  and so sends $(x^i, \xi^j,
\eta^k)$ to $(x^i, \xi^j)$. This projection gives $J^2 M$ the structure of
an affine bundle over $TM$. Indeed, if $(\widetilde{x}^{1},\dots ,\widetilde{%
x}^{n})$ is another system of coordinates around $p$, by taking the second
derivative of a curve expressed in terms of $\widetilde{x}^i$ instead of $x^i
$ we compute the affine relation 
\begin{equation}  \label{fiber}
\widetilde{\eta }^{i}=\frac{\partial \widetilde{x}^{i}}{\partial x^{j}}%
(p)\eta ^{j}+\frac{\partial ^{2}\widetilde{x}^{i}}{\partial x^{j}\partial
x^{k}}(p)\xi ^{j}\xi ^{k}.
\end{equation}
relating the fiber coordinates $\eta$ and $\widetilde \eta$. Finally, if $%
c(t)$ is curve germ passing through $p$ at $t =0$ then $(j^2 c)(0)$ denotes
its second order Taylor expansion, or ``2nd jet'' and is a well-defined
object, independent of coordinates, whose expression in our coordinates is $%
\xi^i = \dot x^i (0)$ and $\eta^j = \ddot x^j (0)$.

Let 
\begin{equation*}
\mathcal{E}_p= \{(j^2 c )(0): c \text{ a solution to the EL equations for $F$
having } c(0) = p \} \subset J^2 _p M.
\end{equation*}
We will give an algebraic description of this subvariety. Due to the
homogeneity of $F$,  this
variety is given by a system of algebraic equations which are homogeneous in
the velocities $\xi$. The Euler-Lagrange equations are 
\begin{equation}
\frac{d}{dt}\Bigl(\frac{\partial F}{\partial \xi ^{k}}\Bigr)-\frac{\partial F%
}{\partial x^{k}}=0.  \label{EL}
\end{equation}%
%
The left-hand side is computed as 
\begin{equation}
\begin{aligned} &\quad\bigl(\omega_l \dot x^l)^{-1}(2g_{kj}\ddot
x^j+2(\partial_m g_{kj})\dot x^m\dot x^j-(\partial_k g_{ij})\dot x^i\dot
x^j\bigr) \\ &+\bigl(\omega_l \dot x^l)^{-2}(-2g_{ki}\omega_j\dot x^i\ddot
x^j-2\omega_k g_{ij}\dot x^i\ddot x^j\\ &\quad\quad\quad\quad\quad
-\omega_k(\partial_m g_{ij})\dot x^m\dot x^i \dot x^j
-2g_{kj}(\partial_m\omega_l)\dot x^j\dot x^m\dot x^l \\
&\quad\quad\quad\quad\quad-g_{ij}(\partial_m\omega_k)\dot x^i\dot x^j\dot
x^m+g_{ij}(\partial_k\omega_l) \dot x^i\dot x^j\dot x^l \bigr) \\
&+(\omega_l \dot x^l)^{-3} \bigl(2g_{il}\omega_k\omega_j\dot x^i\dot
x^l\ddot x^j+2g_{ij}\omega_k(\partial_m\omega_l) \dot x^i\dot x^j\dot
x^m\dot x^l \bigr). \end{aligned}  \label{EL-formula}
\end{equation}%
Replace $\dot{x}^{i}$, $\ddot{x}^{i}$ in \eqref{EL-formula} by $\xi ^{i}$, $%
\eta ^{i}$ respectively, and evaluate all coefficients at $p$. The  result
is 
\begin{equation*}
(\omega_l\xi^l)^{-1}  2(A_{kj}(\xi) \eta ^{j}-b_{k}(\xi)) ,
\end{equation*}%
where 
\begin{equation}
\begin{aligned} A_{kj}(\xi)&=g_{kj}- (\omega_l\xi^l)^{-1}
(g_{ki}\omega_j\xi^i+\omega_k g_{ij}\xi^i)+ (\omega_l\xi^l)^{-2}
g_{il}\omega_k\omega_j\xi^i\xi^l, \\ 
 2b_k (\xi)&=-2(\partial_m
g_{kj})\xi^m\xi^j+(\partial_k g_{ij})\xi^i\xi^j \\
&\quad+(\omega_l\xi^l)^{-1}\bigl(\omega_k(\partial_m g_{ij})\xi^m\xi^i \xi^j
+2g_{kj}(\partial_m\omega_l)\xi^j\xi^m\xi^l \\
&\quad\quad\quad\quad\quad\quad\quad+g_{ij}(\partial_m\omega_k)\xi^i\xi^j%
\xi^m-g_{ij}(\partial_k\omega_l) \xi^i\xi^j\xi^l \bigr) \\ &\quad-2
(\omega_l\xi^l)^{-2} g_{ij}\omega_k(\partial_m\omega_l)\xi^i\xi^j\xi^m\xi^l.
\end{aligned}  \label{A-b}
\end{equation}

\begin{lemma}
\label{kernel} \label{ker-A} If $\xi\notin H$, then $\mathrm{ker} (A_{kj}
(\xi))=\mathbb{R}\xi$.
\end{lemma}

\begin{proof}
One can readily check by hand that $A_{kj} (\xi) \xi^j=0$. Since $T_p M = H
\oplus \mathbb{R} \xi$, the lemma is proved once we show that if $v \in H$
satisfies $A_{kj} v^j = 0$ then $v = 0$. So let $v \in H$ satisfy $A_{kj}
v^j = 0$. Choose any other $w \in H$. A glance at the expression for $A$
shows that $A_{kj} w^k v^j = g_{kj} w^k v^j$. We thus have that $v$ is $g$%
-perpendicular to all of $H$. The non-degeneracy of $g$ restricted to $H$
implies that $v = 0$.
\end{proof}

\begin{remark}
The fact that $A_{kj} (\xi) \xi^j=0$ follows directly from the
parameterization invariance of $\int F$. Indeed, whenever $\gamma(t)$ is an
extremal and $\lambda (t)$ is a reparameterization of the time interval, $%
\gamma (\lambda (t))$ is an extremal. Differentiating, we find that if $%
\xi^i, \eta^j$ are the 2-jets of a solution, then so is $\tilde \xi ^i =
\alpha \xi^i, \tilde \eta^j = \alpha ^2 \eta^j + \beta \xi^i$ where $\alpha
= \dot \lambda$, $\beta = \ddot \lambda$. Now $\xi^i, \eta^j$ satisfy 
$A_{kj}
(\xi) \eta^j - b_k (\xi) = 0$, as does $\tilde \xi^i, \tilde \eta^j$. A bit
of algebra, combined with the fact that $A$ and $b$ are homogeneous of
degree $0$, $2$, in $\xi$ yields $A_{kj} \xi^j = 0$.
\end{remark}

\begin{proposition}
In terms of the fiber coordinates $(\xi, \eta)$ we have 
\begin{equation*}
\mathcal{E}_p =\{(\xi ,\eta )\in \mathbb{R}^{n}\times \mathbb{R}^{n}\ |\
\omega _{l}\xi ^{l}\neq 0, A_{kj}(\xi)\eta ^{j}=b_{k}(\xi)\}.
\end{equation*}
\end{proposition}

\begin{proof}
Denote the right-hand side by $\mathcal{E}^\prime$. We have $\mathcal{E}_p
\subset\mathcal{E}^\prime$ by the above computation of the Euler-Lagrange
equation. Conversely, suppose that $(\xi, \eta)\in\mathcal{E}^\prime$. By
Corollary \ref{geodesic}, there exists a geodesic $x(t)$ of $F$ such that $%
x(0)=p$, $\dot x(0)=\xi$. Then $(\xi, \ddot x(0))\in\mathcal{E}^\prime$, so by
Lemma \ref{ker-A}, we have $\eta=\ddot x(0)+c\, \xi$ for some constant $c$.
If we take a reparametrization $x^\prime(t)=x(\varphi(t))$ such that $%
\dot\varphi(0)=1$, $\ddot\varphi(0)=c$, then $({\dot x}^\prime(0), {\ddot x}%
^\prime(0))=(\xi, \eta)$, so we have $(\xi, \eta)\in\mathcal{E}_p$.
\end{proof}

We now proceed to show that the 2nd derivatives of solutions blow up as
their initial conditions approach $H = {\rm ker}\,\omega$. To this end, write $pr:
J^2 _p M \to T_p M$ for the canonical projection and set 
\begin{equation*}
\ell_{\xi }:=\mathcal{E}_p\cap pr^{-1}(\xi ).
\end{equation*}
By Lemma \ref{kernel}, for $\xi \not\in H$ we have that $\ell_{\xi }$ is an
affine line whose direction is $\xi$. The line $\ell_{\xi}$ lives in $J^2 _p M$
which is diffeomorphic to $\mathbb{R}^n \times \mathbb{R}^n$. We now imagine 
$\xi$ varying so that $\xi \to \xi_0$ We will say that $\ell_{\xi} \to \infty$
as $\xi \to \xi_0$, if, as $\xi \to \xi_0$, the line $\ell_{\xi}$ leaves every
compact set of $J^2 _p M$.

\begin{lemma}
\label{infinity}  There is a Zariski dense open subset $H {\setminus 
\mathcal{N}}$ of $H$ such that for all $\xi_0$ in this set we have $%
\ell_{\xi} \to \infty$ whenever $\xi \to \xi_0$.
\end{lemma}

Before giving the proof we describe the Zariski dense subset of the lemma.
According to Lemma \ref{modify}, we may assume that $g$ is non-degenerate.
We set 
\begin{equation*}
\mathcal{N}:=\{\xi \in T_{p}M\ |\ |\xi |_{g}^{2}=0\ \mathrm{or}\ (\xi
\lrcorner d\omega )|_H=0\}.
\end{equation*}%
Since $H $ is non-integrable, $H{\setminus \mathcal{N}}$ is a Zariski open
dense subset of $H$.

\begin{proof}
We will make use of the formula 
\begin{equation*}
\pi (v )=v -\frac{\omega (v)}{|\omega |_{g}^{2}}\omega _{g}
\end{equation*}
for the $g$-orthogonal projection $\pi :T_{p}M\rightarrow H$. (Note that $%
|\omega |_{g}^{2}\neq 0$ since if $|\omega |_{g}^{2}=0$, we have $\omega
_{g}:=(\omega ^{i})\in H = \mathrm{ker}\,\omega $ and $g(\omega _{g},v )=0$
for all $v \in H $, contradicting to the assumption that $g|_H$ is
non-degenerate.)

Since the assertion of the lemma is independent of coordinates, it suffices
to establish the lemma in one choice of coordinate systems. We will work in
the normal (=geodesic) coordinate system of $g$ centered at $p$. Recall that
these coordinates are specified by the condition that $\partial_k g_{ij}=0$
at $p$ for all $k,i,j$. Let $\nabla$ be the Levi-Civita connection of $g$.
We define 
\begin{align*}
(\nabla\omega)(\xi_1, \xi_2)&:=(\nabla_i \omega_j)\xi_1^i\xi_2^j , \\
(d\omega)_g(\xi)^i&:=\frac{1}{2}(\nabla_j\omega^i-\nabla^i\omega_j)\xi^j.
\end{align*}
Then $g((d\omega)_g(\xi_1), \xi_2)=d\omega(\xi_1, \xi_2)$. In these
coordinates the endomorphism $A=(A_k{}^j)$ and the vector $b=(b^i)$ are
written as 
\begin{equation}  \label{A-b-g}
\begin{aligned} A\eta&=\eta-\frac{\omega(\eta)}{\omega(\xi)}\xi-\frac{g(\xi,
\eta)}{\omega(\xi)}\omega_g+
\frac{|\xi|_g^2\omega(\eta)}{\omega(\xi)^2}\omega_g, \\
b&=\frac{(\nabla\omega)(\xi,
\xi)}{\omega(\xi)}\xi+\frac{|\xi|_g^2}{\omega(\xi)}(d\omega)_g(\xi)
-\frac{|\xi|_g^2 (\nabla\omega)(\xi, \xi)}{\omega(\xi)^2}\omega_g. 
\end{aligned}
\end{equation}
Since $g(\xi, b)=0$, we have 
\begin{equation*}
g(\xi, \pi(b))=-\frac{\omega(\xi)\omega(b)}{|\omega|^2_g}.
\end{equation*}
Then, it follows that 
\begin{align*}
A\pi(b)&=\pi(b)-\frac{g(\xi, \pi(b))}{\omega(\xi)}\omega_g \\
&=\pi(b)+\frac{\omega(b)}{|\omega|_g^2}\omega_g \\
&=b.
\end{align*}
Hence the solution space to $ A(\xi) \eta=b(\xi)$ is 
$ \pi(b(\xi))+\mathbb{R}\xi$. 
In other words 
\begin{equation*}
\ell_{\xi} =  \pi(b(\xi))+\mathbb{R}\xi
\end{equation*}
as a subspace of the affine $\mathbb{R}^n$ coordinatized by $\eta$.

Now choose any Euclidean structure $\mathring{g}$ on this $\mathbb{R}^n$.
There is a unique point $\eta_* \in \ell_{\xi}$ closest to the origin
relative to this metric:  $|\eta_*|_{\mathring g}=\mathrm{dist}_{\mathring
g}(0, \ell_\xi)$. We call this point the \textit{optimal solution}. Our proof
will be complete by showing that $|\eta_*|_{\mathring g} \to \infty$ as $\xi
\to \xi_0 \in H {\setminus {\mathcal{N}}}$.

Since 
\begin{align*}
\pi(b)&= 
\frac{|\xi|_g^2}{\omega(\xi)} \pi((d\omega)_g(\xi))
+\frac{(\nabla\omega)(\xi, \xi)}{\omega(\xi)}\pi(\xi)  \\
&=\frac{|\xi|_g^2}{\omega(\xi)} \pi((d\omega)_g(\xi)) 
+\frac{(\nabla\omega)(\xi, \xi)}{\omega(\xi)}\xi
-\frac{(\nabla\omega)(\xi, \xi)}{|\omega|_g^2}\omega_g,
\end{align*}
we have 
\begin{equation}  \label{optimal-eta}
\begin{aligned}  \eta_* &=\pi(b)-\frac{\mathring g(\xi,
\pi(b))}{|\xi|_{\mathring g}^2}\xi \\ &=\frac{|\xi|^2_g}{\omega(\xi)}
\Bigl(\pi(d\omega)_g(\xi)-\frac{\mathring g(\xi,
\pi(d\omega)_g(\xi))}{|\xi|^2_{\mathring g}}\xi\Bigr) \\ &\quad\quad
-\frac{(\nabla\omega)(\xi, \xi)}{|\omega|_g^2}
\Bigl(\omega_g-\frac{\mathring g(\xi, \omega_g)}{|\xi|_{\mathring
g}^2}\xi\Bigr). \end{aligned}
\end{equation}
The second term is bounded as $\xi\to\xi_0$. To see that the first term
diverges, we will show that 
\begin{equation*}
\pi(d\omega)_g(\xi_0)-\frac{\mathring g(\xi_0, \pi(d\omega)_g(\xi_0))}{%
|\xi_0|^2_{\mathring g}}\xi_0\neq0.
\end{equation*}
Suppose the equality holds. Then, if 
\begin{equation*}
c:=\frac{\mathring g(\xi_0, \pi(d\omega)_g(\xi_0))}{|\xi_0|^2_{\mathring g}}%
=0,
\end{equation*}
we have $\pi(d\omega)_g(\xi_0)=0$. Taking the inner product of $%
\pi(d\omega)_g(\xi_0)$ with any $v \in H$ yields $d \omega (v, \xi_0) = 0$
which contradicts the assumption $\xi_0\not\in\mathcal{N}$. On the other
hand, if $c\neq0$, we have $\pi(d\omega)_g(\xi_0) = c \xi_0$ and taking $g$%
-inner products of both sides of this equation with $\xi_0$ yields $%
d\omega(\xi_0, \xi_0)=c g( \xi_0, \xi_0)$,or $0 = |\xi_0 |_g ^2$ which again
contradicts $\xi_0\not\in\mathcal{N}$. Thus we have $|\eta_*|_{\mathring
g}\to\infty$.
\end{proof}

We note that the optimal solution $\eta_*$ is represented by a rational
function of $\xi$ in any coordinate system $(x^1, \dots , x^n)$, and it is
regular on $T_p M {\setminus H}$.

\begin{corollary}
\label{cor:1} Under the assumptions of this section, there is no affine
connection $\nabla$ on $M$ whose geodesics agree with the $F$-geodesics.
\end{corollary}

\begin{proof}
Let $\Gamma^i _{jk}$ be the Christoffel symbols of an affine connection $%
\nabla$ defined near $p$ so that the $\nabla$-geodesic equations read $\ddot
x^i + \Gamma^i _{jk} \dot x^j\dot x^k = 0$. Let ${\mathcal{F}}_p$ be the
space of 2nd jets of unparameterized geodesics through $p$. It will suffice
to show that ${\mathcal{F}}_p \ne {\mathcal{E}}_p$ where ${\mathcal{E}}_p$
is the space of 2nd jets of unparameterized Kropina geodesics, as before. In
parallel to what we did just above,  $\ell ^{\nabla} _{\xi} = pr^{-1}(\xi )
\cap {\mathcal{F}}_p$ be the affine line in $J^2 M_p$ corresponding to all
the (unparameterized) $\nabla$-geodesics whose first derivative at $p$ is $%
\xi$. Then $\tilde \ell_{\xi}$ consists of the affine line parameterized as $%
\eta(\lambda) = \eta_0 (\xi) +\lambda \xi$ where $\eta_0 ^i (\xi) = -
\Gamma^i _{jk} \xi^j \xi^k$, and where the parameter $\lambda$ of the line
arises by reparameterizing $\nabla$-geodesics and can be identified with the
second derivative of the time reparameterization. Since $\eta_0 (\xi) \to
\infty$ if and only if $\xi \to \infty$ we see that the lines $\tilde
\ell_{\xi}$ cannot go to infinity as $\xi \to \xi_0 \in H {\setminus 
\mathcal{N}}$. It follows that these two sets of lines disagree in some
neighborhood of $H$ and hence that ${\mathcal{E}}_p \ne {\mathcal{F}}_p$.
\end{proof}

\begin{proposition}
Let $F=g/\omega$, $\widehat F=\widehat g/\widehat \omega$ be two Kropina
metrics on $M$ such that

\begin{itemize}
\item $\mathrm{ker}\,\omega$ is non-integrable at every point;

\item $g, \widehat g$ are non-degenerate on $\mathrm{ker}\,\omega$, $\mathrm{%
ker}\,\widehat\omega$ respectively.
\end{itemize}

If $F$ and $\widehat F$ are projectively equivalent, then $\mathrm{ker}%
\,\omega=\mathrm{ker}\,\widehat\omega$. In particular, $%
\mathrm{ker}\,\widehat\omega$ is also non-integrable at every point.
\end{proposition}

\begin{proof}
By Lemma \ref{modify}, we may assume that $g, \widehat g$ are
pseudo-Riemannian. Suppose that there exists a point $p\in M$ with $\mathrm{%
ker}\,\omega_p\neq\mathrm{ker}\,\widehat\omega_p$. We consider the subsets $%
\mathcal{E}, \widehat{\mathcal{E}}\subset \mathcal{A}$ defined from
solutions to the Euler-Lagrange equations of $F, \widehat F$. Choose a
Euclidean structure $\mathring g$ on $T_p M$ and a coordinate system $(x^1,
\dots, x^n)$, and let $\eta_*, \widehat\eta_*$ be the optimal solutions. By
the assumption, there is a nonempty open subset $U\subset\mathbb{R}^n$ such
that $\mathcal{E}\cap (U\times\mathbb{R}^n)=\widehat{\mathcal{E}}\cap
(U\times\mathbb{R}^n)$. Since $\eta_*(\xi)=\widehat\eta_*(\xi)$ on $U$ and
both are rational functions of $\xi$, they coincide globally. In particular,
their poles must coincide. But we know that the closure of the locus of
poles of $\eta_*$ is $\mathrm{ker}\,\omega$. Then, the closure of the locus
of poles of $\widehat\eta_*$ is $\mathrm{ker}\,\omega$. Hence the two
kernels are equal, and in particular the distribution $\mathrm{ker}%
\,\widehat\omega$ is also non-integrable.
\end{proof}

Next we will show that $g$ and $\widehat{g}$ are conformally related on $%
\mathrm{ker}\,\omega =\mathrm{ker}\,\widehat{\omega }$. We consider $\eta
_{\ast },\widehat{\eta }_{\ast }$ in normal coordinates of $g,\widehat{g}$,
and identify them with elements of $T_{p}M$. Then, by \eqref{fiber} we have 
\begin{equation*}
\widehat{\eta }_{\ast }=\eta _{\ast }+O(1)
\end{equation*}%
as $\xi \rightarrow \xi _{0}\in \mathrm{ker}\,\omega {\setminus (\mathcal{N}%
\cup \widehat{\mathcal{N}})}$. It follows from \eqref{optimal-eta} that 
\begin{equation}  \label{limit}
\begin{aligned} &|\xi_0|_{\widehat g}^2\widehat\pi(d\omega)_{\widehat
g}(\xi_0)-\frac{|\xi_0|_{\widehat g}^2}{|\xi_0|_{\mathring g}^2} \mathring
g(\xi_0, \widehat\pi(d\omega)_{\widehat g}(\xi_0))\xi_0 \\
&=|\xi_0|_g^2\pi(d\omega)_g(\xi_0)-\frac{|\xi_0|_g^2}{|\xi_0|_{\mathring
g}^2} \mathring g(\xi_0, \pi(d\omega)_g(\xi_0))\xi_0. \end{aligned}
\end{equation}%
Let us denote the restrictions of $\mathring{g},g,\widehat{g},d\omega $ to $%
\mathrm{ker}\,\omega $ respectively by 
\begin{equation*}
\mathring{h}_{\alpha \beta },\quad h_{\alpha \beta },\quad \widehat{h}%
_{\alpha \beta },\quad B_{\alpha \beta },
\end{equation*}%
where Greek indices run from 1 to $n-1$. Note that $B_{\alpha \beta }$ is
skew symmetric and $(B_{\alpha \beta })\neq 0$ by non-integrability. Since %
\eqref{limit} holds for any $\xi _{0}\in \mathrm{ker}\,\omega {\setminus (%
\mathcal{N}\cup \widehat{\mathcal{N}})}$, we have 
\begin{equation}
\begin{aligned} &\mathring h_{(\sigma\rho}\widehat h_{\alpha\beta}\widehat
h^{\mu\nu}B_{\gamma)\nu} -\widehat h_{(\sigma\rho}\mathring
h_{\alpha|\nu|}\widehat h^{\nu\tau} B_{\beta|\tau|}\delta_{\gamma)}{}^\mu \\
&=\mathring h_{(\sigma\rho} h_{\alpha\beta} h^{\mu\nu}B_{\gamma)\nu} -
h_{(\sigma\rho}\mathring h_{\alpha|\nu|}
h^{\nu\tau}B_{\beta|\tau|}\delta_{\gamma)}{}^\mu, \end{aligned}
\label{index}
\end{equation}%
where $(\cdots )$ denotes the symmetrization over the indices $\sigma ,\rho
,\alpha ,\beta ,\gamma $. Taking the trace of $\mu ,\gamma $ gives 
\begin{equation*}
-(n+1)\widehat{h}_{(\sigma \rho }\mathring{h}_{\alpha |\nu |}\widehat{h}%
^{\nu \tau }B_{\beta )\tau }=-(n+1)h_{(\sigma \rho }\mathring{h}_{\alpha
|\nu |}h^{\nu \tau }B_{\beta )\tau }.
\end{equation*}%
Thus we have 
\begin{equation*}
\widehat{h}_{(\sigma \rho }\mathring{h}_{\alpha |\nu |}\widehat{h}^{\nu \tau
}B_{\beta |\tau |}\delta _{\gamma )}{}^{\mu }=h_{(\sigma \rho }\mathring{h}%
_{\alpha |\nu |}h^{\nu \tau }B_{\beta |\tau |}\delta _{\gamma )}{}^{\mu }
\end{equation*}%
and hence, by \eqref{index}, 
\begin{equation}
(\widehat{h}_{\alpha \beta }\xi ^{\alpha }\xi ^{\beta })\widehat{h}^{\mu \nu
}B_{\gamma \nu }\xi ^{\gamma }=(h_{\alpha \beta }\xi ^{\alpha }\xi ^{\beta
})h^{\mu \nu }B_{\gamma \nu }\xi ^{\gamma }  \label{conformal}
\end{equation}%
for all $\xi \in \mathrm{ker}\,\omega $.

First, we consider the case $n\ge4$. Looking at a nonzero row in the matrix 
$\widehat h^{\mu\nu}B_{\gamma\nu}$, we have 
\begin{equation*}
(\widehat h_{\alpha\beta}\xi^\alpha\xi^\beta)a_\gamma \xi^\gamma
=(h_{\alpha\beta}\xi^\alpha\xi^\beta)b_\gamma \xi^\gamma
\end{equation*}
for some $(a_\gamma)\neq0$ and $(b_\gamma)\neq0$. We regard this as an
equality in the polynomial ring $\mathbb{R}[\xi^1, \dots, \xi^{n-1}]$. Since 
$a_\gamma\xi^\gamma$ is irreducible, it divides either $h_{\alpha\beta}\xi^%
\alpha\xi^\beta$ or $b_\gamma \xi^\gamma$. In the former case, we have 
\begin{equation*}
h_{\alpha\beta}=a_{(\alpha} c_{\beta)}
\end{equation*}
for some $(c_\gamma)$. Since $\dim\mathrm{ker}\,\omega\ge3$, there exists $%
\xi_0\in\mathrm{ker}\,\omega{\setminus\{0\}}$ such that $a_\gamma\xi_0^%
\gamma=c_{\gamma}\xi_0^\gamma=0$, and hence $h_{\alpha\beta}\xi_0^\beta=0$.
This is contradiction since $h_{\alpha\beta}$ is non-degenerate. In the
latter case, we have $b_\gamma=c a_\gamma$ for some constant $c$ and $%
\widehat h_{\alpha\beta}=ch_{\alpha\beta}$.

Next, let $n=3$. In this case, since $\dim\mathrm{ker}\,\omega=2$,
non-integrability condition is equivalent to the contact condition and $B$
is non-degenerate. Then, the equation \eqref{conformal} implies that $%
(\widehat h^{-1}B)(h^{-1}B)^{-1}$ is a scalar multiple of identity, since it
is represented by a diagonal matrix in any basis of $\mathrm{ker}\,\omega$.
Thus $\widehat h$ is a multiple of $h$.

As a result
\begin{equation*}
\widehat{g}=\alpha g+2\beta \cdot \omega
\end{equation*}%
with a function $\alpha $ and a 1-form $\beta $.

Thus we complete the proof of Theorem \ref{thm:2} if we prove the following
proposition:

\begin{proposition}
Let $F=g/\omega ,\widehat{F}=\widehat{g}/\omega $ be two projectively
equivalent Kropina metrics on $M$. Suppose that there exist a function $%
\alpha $ and a 1-form $\beta $ such that $\widehat{g}=\alpha g+2\beta \cdot
\omega $. Then $\alpha $ is constant and $\beta $ is closed.
\end{proposition}

\begin{proof}
We take local coordinates around a point $p\in M$ and consider the defining
equations of $\mathcal{E}$ and $\widehat{\mathcal{E}}$. Since $\widehat{g}%
_{ij}=\alpha g_{ij}+\beta _{i}\omega _{j}+\beta _{j}\omega _{i}$, a direct
computation using \eqref{A-b} gives 
\begin{align*}
 2(\widehat{A}_{kj}\eta ^{j}-\widehat{b}_{k})-2\alpha(A_{kj}\eta
^{j}-b_{k})& =2\omega _{l}\xi ^{l}(\partial _{m}\beta _{k}-\partial
_{k}\beta _{m})\xi ^{m} \\
& \quad +2(\partial _{m}\alpha )g_{kj}\xi ^{m}\xi ^{j}-(\partial _{k}\alpha
)g_{ij}\xi ^{i}\xi ^{j} \\
& \quad -(\omega _{l}\xi ^{l})^{-1}\omega _{k}(\partial _{m}\alpha
)g_{ij}\xi ^{m}\xi ^{i}\xi ^{j}.
\end{align*}%
By the assumption, the equations $ A\eta =b$ and 
$ \widehat{A}\eta =\widehat{b}$ have the same solutions, so the right-hand side of the above equation
must vanish for any $\xi $. Multiplying it by $\omega _{l}\xi ^{l}$ and
taking the limit $\xi \rightarrow \xi _{0}\in \mathrm{ker}\,\omega $, we
have 
\begin{equation*}
\omega _{k}|\xi _{0}|_{g}^{2}\,\xi _{0}^{m}(\partial _{m}\alpha )=0.
\end{equation*}%
Thus, the derivatives of $\alpha $ in the directions of $\mathrm{ker}%
\,\omega $ vanish at any point. Since $\mathrm{ker}\,\omega $ is
non-integrable, this implies that $\alpha $ is constant. Then it follows
that $\partial _{m}\beta _{k}-\partial _{k}\beta _{m}=0$, so $\beta $ is
closed.
\end{proof}

Finally we consider the case where $\omega$ is integrable, and prove Theorem %
\ref{connection}. In this case, we may assume that $\omega$ is closed by
multiplying $g$ and $\omega$ by a function. We note that in the
non-integrable case, geodesics of the Kropina metric are never given by
geodesics of an affine connection due to the singularity described in Lemma %
\ref{infinity}. In contrast, we have the following theorem:

\begin{theorem}
\label{thm:con} Let $g$ be a pseudo-Riemannian metric and $\omega$ a closed
1-form on $M$. Let $\nabla$ be the affine connection whose Christoffel
symbols are given by 
\begin{equation}\label{cn}
\Gamma_{ij}{}^k=\Gamma^g_{ij}{}^k+\frac{1}{ |\omega|_g^2}(\nabla^g_i\omega_j
)\omega^k,
\end{equation}
where $\nabla^g$ is the Levi-Civita connection of $g$ and $\Gamma^g_{ij}{}^k$
are its Christoffel symbols. Then, any geodesic of the Kropina metric $%
F=g/\omega$ is a geodesic of $\nabla$.
\end{theorem}
Note that the last term in \eqref{cn} is a $(1, 2)$ tensor field  symmetric in the lower indices so 
$\nabla$  is a torsion free affine connection.  
\begin{proof}
Let $\gamma(t)$ be a geodesic of $\nabla$. We will show that $\gamma$
satisfies the Euler-Lagrange equation \eqref{EL} for $F$. Take an arbitrary
point $p\in\gamma$. For simplicity, we let $p=\gamma(0)$. Take normal
coordinates $(x^1, \dots, x^n)$ of $g$ centered at $p$. Then, a curve $
(x^i(t))$ satisfies \eqref{EL} at $p$ if and only if $(\xi, \eta)=(\dot
x^i(0), \ddot x^i(0))$ satisfies $ A\eta=b$ with $A, b$ given by 
\eqref{A-b-g}. We recall that one of the solutions is given by 
$\eta= \pi (b)$, and using $d\omega=0$, we have 
\begin{equation*}
\pi(b)=\frac{(\nabla^g\omega)(\xi, \xi)}{\omega(\xi)}\xi-\frac{
(\nabla^g\omega)(\xi, \xi)}{|\omega|_g^2}\omega_g.
\end{equation*}
Since $\xi\in\mathrm{ker}\, A$, 
\begin{equation*}
-\frac{(\nabla^g\omega)(\xi, \xi)}{ |\omega|_g^2}\omega_g
\end{equation*}
is also a solution. On the other hand, the geodesic equation for $\gamma$
implies 
\begin{equation*}
\ddot\gamma(0)+\frac{(\nabla^g\omega)(\dot\gamma(0), \dot\gamma(0))}{
 |\omega|_g^2}\omega_g=0.
\end{equation*}
Hence $\gamma(t)$ solves \eqref{EL}.
\end{proof}

\section{Kropina metric for chains on CR manifolds}

\label{chain-Kropina}

\subsection{CR manifolds}

\label{CR} Let $M$ be a $C^{\infty }$-manifold of dimension $2n+1$. A 
\textit{CR structure} $(H,J)$ is a corank-1 distribution $H\subset TM$
together with a complex structure $J\in \Gamma (\mathrm{End}(H))$. We assume
that the CR structure satisfies the integrability condition that $\Gamma
(T^{1,0}M)$ is closed under Lie bracket, where $T^{1,0}M\subset \mathbb{C}H$
is the eigenspace of $J$ corresponding to the eigenvalue $i=\sqrt{-1}$. For
a choice of a 1-form $\theta $ with $\mathrm{ker}\,\theta=H$, we define the 
\textit{Levi form} by 
\begin{equation*}
h_{\theta }(X,Y):=d\theta (X,JY),\quad X,Y\in H.
\end{equation*}
We say the CR structure is \textit{non-degenerate} if its Levi form is
non-degenerate, and we assume  non-degeneracy hereafter. Then $H$ becomes
a contact distribution.

The \textit{Reeb vector field} of $\theta $ is the vector field $T$ on $M$
characterized by the two conditions: 
\begin{equation*}
\theta (T)=1,\quad T\lrcorner \,d\theta =0.
\end{equation*}%
Any local frame for $\mathbb{C}TM$ of the form $\{T,Z_{\alpha },Z_{%
\overline{\alpha }}:=\overline{Z_{\alpha }}\}$ where $\{Z_{\alpha }\}_{1\leq
\alpha \leq n}$ is a local frame for $T^{1,0}M$ is called  an 
\textit{admissible frame}. The dual coframe $\{\theta ,\theta ^{\alpha
},\theta ^{\overline{\alpha }}\}$ is called an \textit{admissible coframe}.
Then we have 
\begin{equation*}
d\theta =ih_{\alpha \overline{\beta }}\theta ^{\alpha }\wedge \theta ^{%
\overline{\beta }}\quad \mathrm{with}\ h_{\alpha \overline{\beta }%
}:=h_{\theta }(Z_{\alpha },Z_{\overline{\beta }}).
\end{equation*}%
Since $\theta$ is contact the matrix of the Levi form $h_{\alpha \overline{%
\beta }} $ is invertible and we can use its inverse $h^{\alpha \overline{%
\beta }}$ to raise indices. If we rescale the contact form as $\widehat{%
\theta }=e^{\Upsilon }\theta $ with $\Upsilon \in C^{\infty }(M)$, the Reeb
vector field transforms as 
\begin{equation*}
e^{\Upsilon }\widehat{T}=T-i\Upsilon ^{\alpha }Z_{\alpha }+i\Upsilon ^{%
\overline{\alpha }}Z_{\overline{\alpha }},
\end{equation*}%
and accordingly the admissible coframe $\{\widehat{\theta },\widehat{\theta }%
^{\alpha },\widehat{\theta }^{\overline{\alpha }}\}$ dual to $\{\widehat{T}%
,Z_{\alpha },Z_{\overline{\alpha }}\}$ satisfies 
\begin{equation}
\widehat{\theta }^{\alpha }=\theta ^{\alpha }+i\Upsilon ^{\alpha }\theta .
\label{admissible}
\end{equation}

A contact form $\theta $ defines a canonical linear connection $\nabla $ on $%
TM$, called the \textit{Tanaka-Webster} connection \cite{Tanaka, Webster}. This connection
preserves $T^{1,0}M$, satisfies $\nabla T=0$, $\nabla h_{\theta }=0$, and
the structure equation 
\begin{equation*}
d\theta ^{\alpha }=\theta ^{\beta }\wedge \omega _{\beta }{}^{\alpha }+A_{%
\overline{\beta }}{}^{\alpha }\theta \wedge \theta ^{\overline{\beta }},
\end{equation*}%
where $\omega _{\beta }{}^{\alpha }$ is the connection 1-form, and $A_{%
\overline{\beta }}{}^{\alpha }$ is a tensor called \textit{Tanaka-Webster
torsion}. If we set $\omega_{\overline\alpha}{}^{\overline\beta}:=\overline{\omega_\alpha{}^\beta}$ and $A_\beta{}^{\overline\alpha}:=\overline{A_{\overline\beta}{}^\alpha}$, then we have 
\[
dh_{\alpha\overline\beta}=\omega_{\alpha\overline\beta}+\omega_{\overline\beta\alpha}, \quad A_{\alpha\beta}=A_{\beta\alpha}.
\]
Taking the trace of the curvature form $d\omega _{\beta
}{}^{\alpha }-\omega _{\beta }{}^{\gamma }\wedge \omega _{\gamma }{}^{\alpha
}$, we obtain the \textit{Tanaka-Webster Ricci tensor} $R_{\beta \overline{%
\gamma }}$: 
\begin{equation}
d\omega _{\alpha }{}^{\alpha }=R_{\beta \overline{\gamma }}\theta ^{\beta
}\wedge \theta ^{\overline{\gamma }}+(\nabla ^{\alpha }A_{\alpha \beta
})\theta ^{\beta }\wedge \theta -(\nabla ^{\overline{\alpha }}A_{\overline{%
\gamma }\overline{\alpha }})\theta ^{\overline{\gamma }}\wedge \theta .
\label{ricci}
\end{equation}%
The \textit{Tanaka-Webster scalar curvature} is defined by $R:=R_{\beta
}{}^{\beta }$. There are transformation formulas for these quantities under
rescaling $\widehat{\theta }=e^{\Upsilon }\theta $. For example, $R$
transforms as 
\begin{equation*}
e^{\Upsilon }\widehat{R}=R+(n+1)\Delta _{b}\Upsilon -n(n+1)\Upsilon _{\alpha
}\Upsilon ^{\alpha },
\end{equation*}%
where $\Delta _{b}:=-\nabla _{\alpha }\nabla ^{\alpha }-\nabla ^{\alpha
}\nabla _{\alpha }$; see e.g., \cite{Lee1}.

\subsection{The Fefferman metric and chains}
Chains on CR manifolds are first introduced by \'E. Cartan \cite{Cartan} in dimension three by using the Cartan connection. In this paper, we define chains via Fefferman's conformal structure by following \cite{fefferman, Lee1}.

The \textit{CR canonical bundle} is the CR version of the canonical line bundle of
complex geometry. It is the complex line bundle over $M$ defined by 
\begin{equation*}
K_{M}:=\{\zeta \in \wedge ^{n+1}\mathbb{C}T^{\ast }M\ |\ \overline{Z}%
\lrcorner \,\zeta =0\ \mathrm{for\ any\ }Z\in T^{1,0}M\}.
\end{equation*}%
Then the \textit{Fefferman space} is defined to be the circle bundle 
\begin{equation*}
\mathcal{C}:=K_{M}^{\ast }/\mathbb{R}_{+},
\end{equation*}%
where $K_{M}^{\ast }=K_{M}\setminus \{\mathrm{zero\ section}\}$.  Fix a 
contact form $\theta$. Each choice of  admissible coframe $\{\theta
,\theta ^{\alpha },\theta ^{\overline{\alpha }}\}$ yields  a local section %
 $\zeta _{0}=[\theta \wedge \theta ^{1}\wedge \cdots \wedge \theta ^{n}]$ for %
$\mathcal{C}$. (Here $[.]$ means the equivalence class under $\mathbb{R}_{+}$) The corresponding fiber coordinate $s\in \mathbb{R}/2\pi 
\mathbb{Z}$ is defined by writing  $\zeta =e^{is}\zeta _{0}\in \mathcal{C}$.

Let $\omega _{\alpha }{}^{\beta }$ be the Tanaka-Webster connection 1-form
with respect to $\{\theta ^{\alpha }\}$. One can see that 
\begin{equation*}
\sigma :=\frac{1}{n+2}\Bigl(ds-\mathrm{Im}\,\omega _{\alpha }{}^{\alpha }-%
\frac{R}{2(n+1)}\theta \Bigr)
\end{equation*}%
is invariant under changes of $\{\theta ^{\alpha }\}$ and defines a global
1-form on $\mathcal{C}$. For another choice of contact form $\widehat{\theta 
}=e^{\Upsilon }\theta $, we have 
\begin{equation}
\widehat{\sigma }=\sigma +\frac{i}{2}(\Upsilon _{\alpha }\theta ^{\alpha
}-\Upsilon _{\overline{\alpha }}\theta ^{\overline{\alpha }})-\frac{1}{2}%
\Upsilon _{\alpha }\Upsilon ^{\alpha }\theta .  \label{sigma}
\end{equation}%
We define the \textit{Fefferman metric} on $\mathcal{C}$ as the
pseudo-Riemannian (Lorentzian if the Levi form $h_{\alpha \overline{\beta }}$
is positive definite) metric 
\begin{equation*}
G:=h_{\alpha \overline{\beta }}\theta ^{\alpha }\cdot \theta ^{\overline{%
\beta }}+2\theta \cdot \sigma .
\end{equation*}%
By using the transformation formulas \eqref{admissible}, \eqref{sigma}, one
has 
\begin{equation*}
\widehat{G}=e^{\Upsilon }G
\end{equation*}%
for $\widehat{\theta }=e^{\Upsilon }\theta $. Thus, the conformal class of
the Fefferman metric is independent of $\theta $. We set 
\begin{equation*}
K:=(n+2)S,
\end{equation*}
where $S$ is the infinitesimal generator of the circle action on $\mathcal{C}$.
 $K$ is a null Killing vector field on $\mathcal{C}$ satisfying $%
K\lrcorner\, G=\theta$.

\begin{definition}
\label{def:4.1} A \textit{chain} is a curve $\gamma (t)$ on $M$ which is the
projection of a null geodesic $\widetilde{\gamma}(t)$ for the Fefferman
metric $G$ which satisfies $G( K, \dot{\widetilde{\gamma}}(0))\neq 0$.
\end{definition}

Note that $\theta (\dot\gamma(t))$ $\neq $ $0$ follows from the definition.
Since null geodesics are conformally invariant, chains are CR invariant
curves on $M$; see e.g., \cite[P. 306]{BFG}.

\subsection{The Kropina metric for the Fefferman metric}

We apply our construction of the Kropina metric to the Fefferman metric.
Let $\zeta_0=\theta\wedge\theta^1\wedge\cdots\wedge\theta^n$ be the section
of $\mathcal{C}$ on an open subset $U\subset M$ defined by an admissible
coframe $\{\theta, \theta^\alpha, \theta^{\overline\alpha}\}$.  
Identify $U$ with the hypersurface $\{s=0\}\subset\mathcal{C}$ 
transverse to $K$. From  definition \eqref{Kropina-Lorentz}, the
corresponding Kropina metric is given by 
\begin{equation}  \label{Kropina-Fefferman}
F_{\theta, \zeta_0} =\frac{h_{\alpha\overline\beta}\,\theta^\alpha\cdot%
\theta^{\overline\beta}}{\theta} -\frac{2}{n+2}\Bigl(\mathrm{Im}\,\omega_%
\alpha{}^\alpha+\frac{R}{2(n+1)}\theta\Bigr).
\end{equation}
By Theorem \ref{fermat}, we have

\begin{theorem}
\label{thm:chainsK} Chains on a non-degenerate CR manifold $M$ are locally
the geodesics of the Kropina metric $F_{\theta ,\zeta _{0}}$ defined by %
\eqref{Kropina-Fefferman}.
\end{theorem}

If we consider the corresponding admissible coframe \eqref{admissible} for
rescaled contact form $\widehat\theta=e^\Upsilon\theta$, we see that $%
\widehat\zeta_0=\zeta_0$. Thus, the hypersurface does not change and it
follows that 
\begin{equation*}
F_{\widehat\theta, \widehat\zeta_0}=F_{\theta, \zeta_0}.
\end{equation*}
On the other hand, if we take another admissible coframe $%
\theta^{\prime\alpha} =\theta^\beta V_\beta{}^\alpha$ for the same $\theta$,
we have $\zeta_0^\prime=|\det(V_\beta{}^\alpha)|^{-1}\det(V_\beta{}%
^\alpha)\zeta_0$ and 
\begin{equation*}
F_{\theta, \zeta_0^\prime}=F_{\theta, \zeta_0}+\frac{1}{n+2}d\,\mathrm{Im}%
\bigl(\log\det(V_\beta{}^\alpha)\bigr).
\end{equation*}
The change corresponds to the fact that we have different hypersurfaces in $%
\mathcal{C}$. This transformation formula implies that we can define a
global Kropina metric for chains if $c_1(T^{1,0}M)=0$ in $H^2(M; \mathbb{R})$%
. In particular, if $M$ admits a pseudo-Einstein contact form, we have a
global Kropina metric, which is described as follows.

We first recall the pseudo-Einstein condition for contact forms; we refer the reader to 
\cite{Lee2, CY, Hirachi} for detail.  A contact form $\theta$ is called \textit{pseudo-Einstein} if
the associated Tanaka-Webster connection satisfies 
\begin{equation}  \label{pseudo-Ein}
R_{\alpha\overline\beta}=\frac{1}{n}Rh_{\alpha\overline\beta}, \quad
\nabla_\alpha R=in\nabla^{\beta}A_{\alpha\beta}.
\end{equation}
When $n\ge2$, the first equation implies the second, and when $n=1$ the
first equation always holds. If $\theta$ is pseudo-Einstein, a new contact
form $\widehat\theta=e^{\Upsilon}\theta$ is pseudo-Einstein if and only if $%
\Upsilon$ is CR pluriharmonic, namely, 
\begin{equation*}
\nabla_\alpha\nabla_{\overline\beta}\Upsilon=\frac{1}{n}(\nabla_\gamma%
\nabla^{\gamma}\Upsilon) h_{\alpha\overline{\beta}}, \quad
\nabla_\alpha\nabla_\beta\nabla^\beta\Upsilon=-inA_{\alpha\beta}\Upsilon^%
\beta.
\end{equation*}
Again, when $n\ge2$ the first equation implies the second, and when $n=1$
the first equation always holds. When $M$ is embedded in $\mathbb{C}^{n+1}$,
Fefferman's defining function $\rho$, which gives an approximate solution to
the complex Monge-Amp\`ere equation \cite{fefferman}, defines a
pseudo-Einstein contact form $\theta=(\mathrm{Im}\, \partial \rho)|_{TM}$.

\begin{lemma}[{\protect\cite[Lemma 4.1]{Lee2}}]
\label{coframe} Let $\theta$ be a pseudo-Einstein contact form. Then, in a
neighborhood of any point on $M$, there exists an admissible coframe such
that $\omega_\alpha{}^\alpha=-(i/n) R\,\theta$.
\end{lemma}

\begin{proof}
By \eqref{ricci} and \eqref{pseudo-Ein}, we have $d(\omega_\alpha{}%
^\alpha+(i/n)R\,\theta)=0$ for any admissible coframe. Then, using the
transformation formula $\omega^\prime_\alpha{}^\alpha=\omega_\alpha{}%
^\alpha-d\log\det(V_\beta{}^\alpha)$ for the change of coframe $%
\theta^{\prime\alpha}=\theta^\beta V_\beta{}^\alpha$, we can construct a
coframe with $\omega_\alpha{}^\alpha+(i/n)R\,\theta=0$.
\end{proof}

We take a local section $\zeta$ of $\mathcal{C}$ defined via a coframe given
in Lemma \ref{coframe} around each point on $M$. Then the set $\{F_{\theta,
\zeta}\}$ defines a global Kropina metric 
\begin{equation}  \label{global-kropina}
F_\theta=\frac{h_{\alpha\overline\beta}\,\theta^\alpha\cdot\theta^{\overline%
\beta}}{\theta} +\frac{R}{n(n+1)}\theta.
\end{equation}
Note that the right-hand side is now defined invariantly for any admissible
coframe $\{\theta^\alpha\}$. If $\widehat\theta=e^\Upsilon\theta$ is another
pseudo-Einstein contact form, we have 
\begin{equation*}
F_{\widehat\theta}=F_\theta-i(\Upsilon_\alpha\theta^\alpha-\Upsilon_{%
\overline\alpha}\theta^{\overline\alpha})+\frac{1}{n}(\Delta_b\Upsilon)
\theta.
\end{equation*}
The CR pluriharmonicity of $\Upsilon$ implies that the 1-form in the
right-hand side is closed ({\cite[Lemma 3.2]{Hirachi}}). When $\theta=(%
\mathrm{Im}\, \partial \rho)|_{TM}$ with Fefferman's defining function $\rho$%
, we can also express the metric as $F_\theta=(\partial\overline\partial%
\rho)/\theta$, where we regard $\partial\overline\partial\rho$ as a
symmetric 2-tensor and restrict it to $M$. \ \newline

By the non-degeneracy of the CR structure, the Kropina metric $F_{\theta
,\zeta _{0}}$ satisfies the assumption of Theorem \ref{thm:2}. As an
immediate consequence of Corollary \ref{cor:1}, we have the following
theorem:

\begin{theorem}[{\protect\cite[Theorem 5.3]{CZ}}]
There is no affine connection on $M$ such that sufficiently big set of its
geodesics are chains of a certain CR structure.
\end{theorem}

We can also prove that a sufficiently big family of chains determines the CR
structure up to conjugacy. \ \newline

\textit{Proof of Corollary \ref{chain-diffeo}}

By Theorem \ref{thm:2}, the contact distributions coincide. Let $\theta$ be
a contact form for $H=\widehat{H}$. Then, the Kropina metrics $F,\widehat{F}$
defined by \eqref{Kropina-Fefferman} are related as $\widehat{F}=cF+\beta $
with a constant $c$ and a closed 1-form $\beta $. If we write $F=g/\theta ,%
\widehat{F}=\widehat{g}/\theta $, we have 
\begin{equation*}
\widehat{g}=cg+\beta \cdot \theta 
\end{equation*}
on $TM{\setminus H}$. By continuity, this also holds on $H$, so we obtain 
\begin{equation*}
d\theta (X,\widehat{J}Y)=c\,d\theta (X,JY),\quad X,Y\in H,
\end{equation*}
which implies $\widehat{J}=cJ$. Since $J$ and $\widehat{J}$ are complex
structures, we have $c=\pm1$. \qed

  \subsection{Pseudo-Einstein contact forms  for the Burns-Shnider example.  } \label{BS}

Corollary  \ref{cr-connect-thm} says that on a connected, compact strictly pseudoconvex CR 
manifold which admits a pseudo-Einstein contact form   with positive 
Tanaka-Webster scalar curvature,   any two points can be joined by a chain.   As we mentioned in the introduction,  in the  
Burns-Shnider's example of a compact, spherical, strictly pseudoconvex CR 
manifold not every two points can be joined by a chain. It follows that  the Tanaka-Webster scalar curvature with respect to the  canonically 
constructed pseudo-Einstein contact form fails to  be strictly positive. 

Let us   examine this contact form to  see how the positivity fails. We
fix a real number $1\neq r>0$ and define an action of $\mathbb{Z}$ on the
Heisenberg group $H=\mathbb{C}^{n}\times \mathbb{R}$ by $m\cdot
(z,t):=(r^{m}z,r^{2m}t)$. The example is given by the compact quotient $%
S^{2n}\times S_{(r)}^{1}:=(H{\setminus \{(0,0)\}})/\mathbb{Z}$. Let 
\begin{equation*}
\theta _{0}:=dt+i\sum_{\alpha =1}^{n}(z^{\alpha }dz^{\overline{\alpha }}-z^{%
\overline{\alpha }}dz^{\alpha })
\end{equation*}%
be the standard contact form on $H$. We set $\theta :=\rho ^{-2}\theta _{0}$%
, where $\rho :=(|z|^{4}+t^{2})^{1/4}$ is the Heisenberg norm. Then since $%
\log \rho $ is CR pluriharmonic, $\theta $ is pseudo-Einstein, and is
invariant under the action of $\mathbb{Z}$, so descends to a pseuo-Einstein
contact form on $S^{2n}\times S_{(r)}^{1}$. We set 
\begin{equation*}
Z_{\alpha }:=\frac{\partial }{\partial z^{\alpha }}+iz^{\overline{\alpha }}%
\frac{\partial }{\partial t}
\end{equation*}%
so that $\{Z_{\alpha }\}$ is a frame for $T^{1,0}H$. The Levi form for $%
\theta _{0}$ with respect to this frame is given by $2\delta _{\alpha 
\overline{\beta }}$. By using the transformation formula in \cite{Lee2}, one
computes the Tanaka-Webster scalar curvature for $\theta $ as 
\begin{align*}
R& =(n+1)\rho ^{2}\delta ^{\alpha \overline{\beta }}\bigl((Z_{\alpha }Z_{%
\overline{\beta }}+Z_{\overline{\beta }}Z_{\alpha })\log \rho -2n(Z_{\alpha
}\log \rho )(Z_{\overline{\beta }}\log \rho )\bigr) \\
& =\frac{n(n+1)|z|^{2}}{2\rho ^{2}}.
\end{align*}%
This is nonnegative, but vanishes on the circle $\{(0,t)\ |\ t\in \mathbb{R}%
^{\ast }\}/\mathbb{Z}$. Thus, the 2-tensor $g$ in the associated Kropina
metric $F=g/\theta $ (see \eqref{global-kropina}) is not globally positive
definite.

On the other hand, observe that $S^{2n}\times S_{(r)}^{1}$ is spherical and
has positive CR Yamabe constant $\lambda = \lambda(n,r)$ less than 
the one for the standard CR sphere of dimension $2n+1$. By a theorem of Jerison and Lee (\cite{JL}), we
can find a contact form (called the Yamabe contact form) with Tanaka-Webster scalar
curvature equal to $\lambda$. From Corollary \ref%
{cr-connect-thm} and the chain nonconnectivity of $S^{2n}\times S_{(r)}^{1}$,
this Yamabe contact form cannot be pseudo-Einstein.

\section{Indicatrix  and Proof of Theorem \ref{connect-thm}}

\label{indicatrix}

Following  the standard usage of Finsler geometry,   the 
\textit{indicatrix} of the Kropina metric $F= g/\omega$ of $M$ at the point $x
\in M$ is  the closure of the locus   $\{\xi\in T_x M: F(x,
\xi) = 1 \} $.    As we will see shortly,
the indicatrix   is a  conic  passing through the zero vector.

Figure \ref%
{fig:indicatrix} shows the  relation between the indicatrix and the lightcone of the corresponding
``Fefferman metric'' $\tilde g = g + 2 \omega dx^0$. This lightcone at a point $\tilde x
\in \widetilde{M}$ over $x$ is the quadric $\tilde g_{\tilde x} (v,v) = 0$
within $T_{\tilde x} \widetilde{M}$. (Since $K$ is Killing, this quadric is
independent of choice of lift $\tilde x$ of $x$.) Projectivize the light
cone to obtain a projective quadric within $\mathbb{P} (T_{\tilde x} 
\widetilde{M})$. View the projectivized quadric in the affine
chart $dx^0 = -1/2$, so that its equation becomes $\tilde g = g -
\omega =0$ which is the equation for   the indicatrix of the Kropina metric.

\begin{figure}[ht]
\scalebox{0.4}{\includegraphics{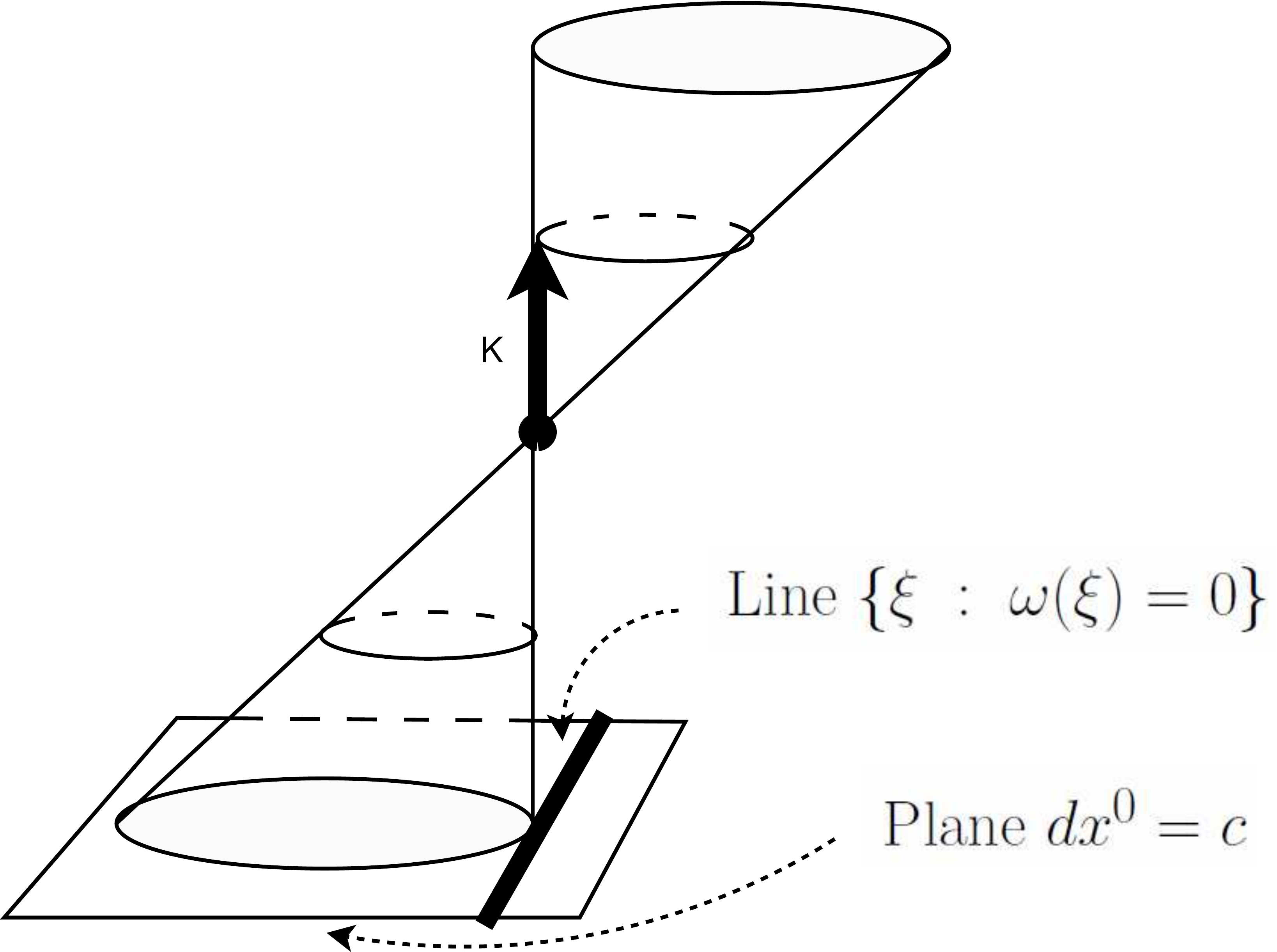}}
\caption{The indicatrix corresponds to the lightcone of the Fefferman
metric. The fact that 0 lies on the indicatrix corresponds to the fact that
the Killing vector $K$ is lightlike.}
\label{fig:indicatrix}
\end{figure}

Direct calculations show that the indicatrix at $x$  is given by 
\begin{equation*}
g_x ( v- W(x) , v- W(x))   = g_x (W(x), W(x)), 
\end{equation*}
where $W(x)$ is  the half of the vector field obtained
by raising the indices of $\omega$ using $g$:  
\begin{equation*}
W^i = \frac{1}{2} g^{ij} \omega_j.
\end{equation*}
If $g$ is positive definite we see that the indicatrix is the   $g$-sphere  of radius $|W(x)| = \sqrt{g_x (W(x),W(x))}$ centered at $W(x)$ within $T_x M$.

Let us   begin    the proof of Theorem \ref{connect-thm}.  We assume that $g/\omega$ is a Kropina metric on a connected $M$ 
with $g$ positive
definite.  We assume that  
$\omega\wedge  d\omega \not= 0$ at every  point $p$ of some  
 connected everywhere dense subset of $M$.

  We  say that an absolutely continuous    curve $c$ on $M$ is {\it admissible},  if  we have   $F(c(t),c'(t))\ge 0$ for almost all $t$.  We allow $F(c(t),c'(t))=\infty$.
 The {\it length} of an   admissible curve $c$ is given 
 by   $ L(c)= \int_c F(c, c')  $.  This length    can be  infinity.  The length is invariant under orientation-preserving re-parameterizations  of the curve.
  
 Define the {\it distance} $d(p,q)$ 
   of two points  $p, q\in M$ 
   to be   the infinum of lengths of all admissible curves connecting $p$ to $q$.   
  In view of  the Chow-Rashevsky theorem  (see \cite{Chow} or \cite[\S 1]{Gromov}),    the  condition $\omega\wedge  d\omega \not= 0$  on a connected everywhere dense subset  implies that  any point $p$  can be  joined to any point $q$ by a regular smooth  curve such that $0<F(c(t), c'(t))\ne \infty $
     for all $t$. This  guarantees  that  the distance function is well defined and finite. This distance clearly  satisfies the triangle inequality.  
      Due to the compactness of 
     the indicatrix,  
     the distance is   positive  for $p\ne q$.   
     However, this distance  is not symmetric in $p$ and  $q$.

  By the  $r-${\it ball}  $B_r(p)$ around $p$ we understand the set $B_r(p):= \{p\in M \ : \ d(p,q)\le r\}$. For any point $p$  and  for 
  small positive $r$  the ball  $B_r(p)$ is  compact.   
By a standard argument (see e.g. \cite[\S 2.5]{BBI} or \cite[\S 6.6]{Bao}),  if $B_r(p)$ is compact,  then for any point $q$ in $B_r(p)$ we have  the existence of 
an arc-length parameterized minimal geodesic $\gamma$  connecting $p$ to $q$.  By definition, this is  a curve $\gamma$  such that
for any two times  $t_1\le t_2$ in its domain we have   $t_2-t_1= d(\gamma(t_1), \gamma(t_2))$.  This  condition 
 implies that the length of $\gamma$ is    the distance from $p$ to $q$ and that for any point $\gamma(t)\ne q$ we have $d(p, \gamma(t))<d(p,q)$.   
 
 Note that though  \cite{BBI} generally  considers  symmetric distance functions (when $d(p,q)= d(q,p)$), symmetry  is not used in  \cite[\S 2.5]{BBI} so the existence of an arc-length parameterized minimal geodesic is also established in our situation.

  If we knew that arc-length parameterized  minimal geodesics were Kropina geodesics as defined in the introduction, 
  we would have proven (A) of the theorem already.  But we do not know this yet.  According to  our original  definition in the introduction,  a Kropina geodesic   must be smooth and  satisfy  $\omega(\gamma') \ne 0$ everywhere,    while an arc-length parameterized minimal geodesic
  may have points  at which  $\gamma'$ is not defined or  $\omega(\gamma') =0$.    Clearly, every arc-length parameterized minimal 
   geodesic  such that $\gamma'(t)$  exists and  satisfies $\omega(\gamma'(t))\ne 0$ for all $t$ 
   is a Kropina geodesic along which   $F(\gamma') \equiv 1>0$.

\begin{figure}[ht]
\scalebox{0.4}{\includegraphics{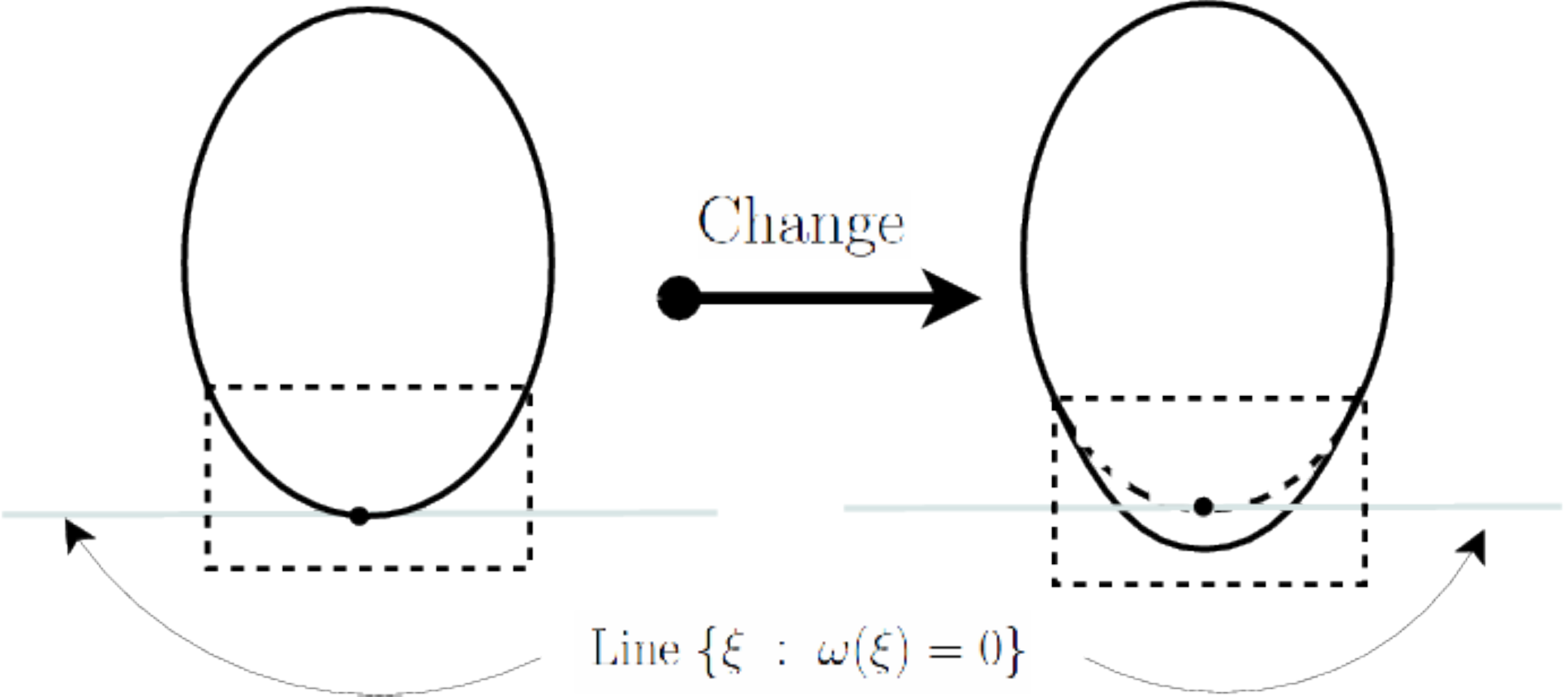}}
\caption{Change of the indicatrix of a Kropina metric such that the result is  the indicatrix of  a Finsler metric. The change is done only in a small box around zero so the Finsler metric   coincides with  the Kropina metric on vectors   with     
$F(v)=1$ and $ \omega(v)>\delta$   for some small  $\delta>0$.}
\label{fig:indicatrix_change}
\end{figure}
 
In order to  prove   that any  arc-length parameterized minimal geodesic  $\gamma$  is  a  Kropina geodesic, 
  observe that for any $\delta>0$ we  can find a smooth (but irreversible: $\tilde F(x,-v) \ne \tilde F(x,v)$)  Finsler metric $\tilde F$  
  such that  $\tilde F = 1$  agrees with $F$  on the locus  $\omega(v)>\delta$, and such that  $\tilde F \le F$ whenever
  $\omega (v) > 0$, see figure \ref{fig:indicatrix_change}.

 By the standard ``Gauss-lemma'' type theorem  of Finsler geometry (see e.g. \cite[\S 6.3]{Bao}), we have that for  any $v  \in T_xM$ with $\tilde F(v)= 1$ 
 there is a unique $\tilde F$-geodesic $\tilde \gamma$ through $x$,  tangent to $v$ at $t=0$, and such that the parameter $t$ is the  $\tilde F$-arclength.
 This  curve is smooth and for all  sufficiently small $\varepsilon$ its restriction   to  $[0,\varepsilon]$ is  the   unique $\tilde F$-shortest curve between 
 $x$ and its endpoint $\tilde \gamma(\varepsilon)$.   Take $\varepsilon_0$ small enough so that $\omega(\tilde \gamma' (t)) > \delta$ for all $t \in [0, \varepsilon_0)$. 
 Since  $F$ and $\tilde F$ agree on  $\{ F=1, \omega > \delta\}$, each  curve $\tilde \gamma|_{[0, \varepsilon_0)}$ is also the  unique  arc-length parameterized minimal $F$-geodesic  joining its   endpoints.  (Use $\tilde F \le F$ and argue by contradiction.)  
 Observe that  $\varepsilon_0$  can be universally  bounded away  from zero by a constant which depends on $\delta$, $\omega(v)$,  and    the 
  first and the second derivatives of $\omega$ and $g$ in some neighborhood of $x$. 
 
Now consider an  arc-length parameterized minimal  geodesic $ \gamma$ for $F$.   We call $t$ from its 
 domain of definition  {\it regular}, if     $\gamma'(t)$ exists  and  $\omega(\gamma'(t)) >0$.     By the paragraph above, the  
 regular times  form an open set of the interval of definition of the curve.  
   The {\it singular} times, being by definition the complement of the regular times,
 form a closed set  of measure zero.  
 
The key step in the proof of   Theorem  \ref{connect-thm} is the following  Lemma:

\begin{lemma}
\label{lem:smooth}

  Let $c:[0, T]\to M$ be an arc-length parameterized   minimal  geodesic of $F$ such that all $t\in [0,T)$ are regular.   Then,  $t=T$ is also 
      regular, that is  $\gamma'(T)$ exists  and $\omega(\gamma'(T))>0$.   
 
\end{lemma}

\begin{proof} As in Section \ref{sec:kropina},  for any $t\in [0, T)$
we set $\xi:= \gamma'(t)$, $\eta:= \gamma''(t)$. The Euler-Lagrange equation for $F$ in normal geodesic coordinates of $g$ around a point $\gamma(t)$ is calculated in  Section \ref{sec:kropina} and is equivalent to the system of equations  
given by $A\eta=b$ with
\begin{equation*}  
\begin{aligned} A\eta&=\eta-\frac{\omega(\eta)}{\omega(\xi)}\xi-\frac{g(\xi,
\eta)}{\omega(\xi)}\omega_g+
\frac{|\xi|_g^2\omega(\eta)}{\omega(\xi)^2}\omega_g, \\
b&=\frac{(\nabla\omega)(\xi,
\xi)}{\omega(\xi)}\xi+\frac{|\xi|_g^2}{\omega(\xi)}(d\omega)_g(\xi)-\frac{|%
\xi|_g^2 (\nabla\omega)(\xi, \xi)}{\omega(\xi)^2}\omega_g. \end{aligned}
\end{equation*}
Recall  that these equations are written for arbitrary parameterized geodesics; let us use the assumption that our geodesic is arc-length parameterized, 
which in regular points means   
\begin{equation}\label{unit}
|\xi|_g^2=\omega(\xi).
\end{equation}
Differentiating this equation in the direction of $\xi$, we obtain  
\begin{equation}\label{g-xi-eta}
2g(\xi, \eta)=\omega(\eta)+(\nabla\omega)(\xi, \xi).
\end{equation}
By \eqref{unit} and \eqref{g-xi-eta}, the equation  $A\eta=b$ is reduced to
\begin{equation*}
\eta-\frac{\omega(\eta)}{\omega(\xi)}\xi+\frac{\omega(\eta)}{2\omega(\xi)}
\omega_g=\frac{(\nabla\omega)(\xi, \xi)}{\omega(\xi)}
\Bigl(\xi-\frac{1}{2}\omega_g\Bigr)+(d\omega)_g(\xi).
\end{equation*}
Taking $\omega$ of both sides gives
$$
\omega(\eta)=
\frac{2(\nabla\omega)(\xi, \xi)+2d\omega(\xi, \omega_g)}{|\omega|_g^2}\omega(\xi)
-(\nabla\omega)(\xi, \xi).
$$
Thus we have
\begin{align*}
\frac{d}{dt}\omega(\xi)&=\omega(\eta)+(\nabla\omega)(\xi, \xi) \\
&=\frac{2(\nabla\omega)(\xi, \xi)+2d\omega(\xi, \omega_g)}{|\omega|_g^2}\omega(\xi) \\
&\ge -C\omega(\xi)
\end{align*}
where the constant $C$ is obtained by taking the maximum of  the absolute value 
of the coefficient $(2(\nabla\omega_x )(\xi, \xi)+2d\omega_x (\xi, \omega_g(x)))/(|\omega(x)|_g^2)$ appearing in the second line,
the maximum being taken as 
 $x$ varies over the compact curve $c([0,T])$ and $\xi$ varies over  
over the compact indicatrix  at $x$. (Recall that  $F(\gamma(t), \gamma'(t))= 1$ for $t<T$.)  
Hence we obtain for all $t \in [0, T)$  
$$
\omega(\xi)\ge\omega({ \gamma^\prime}(0))e^{-Ct} >\omega({ \gamma^\prime}(0))e^{-CT}=:\delta .    
$$ 
 As   explained above,  there exists      a  universal   positive 
  constant 
   $\varepsilon_0$ depending  on $\delta$ and on behavior of $g$ and $\omega$ in some neighborhood  of the geodesic   such that  for each    $t\in [0,T)$  the geodesic  can be extended for  time at least $\varepsilon_0$.   In particular,  $t=T$ is regular. 
  \end{proof} 

 We return to the proof of Theorem \ref{connect-thm}. 
 We will connect $p$ to a point $q$   by a Kropina geodesic   in the interior  of  the ball $B_r(p)$. 
 If we work under the assumptions of part (A) of the theorem, we take  $r>0$  sufficiently small. In view  
 of the condition $\omega\wedge d \omega_p\ne 0$,    the set of  interior points of $B_r(p)$ is   an   open   neighborhood of $p$. 
If we work under the assumptions of part (B), we take $r$ sufficiently large   so that  the set of interior points of $B_r(p)$ coincides with  the whole $M$.  
 
 For any interior  point $q$ of $B_r(p)$ we have $d(p,q)<r$. Consider 
  an  arc-length parameterized  minimal geodesic  $\gamma:[0,d(p,q)]\to M$ 
 joining $p$ to $q$. The existence of $\gamma$ is  explained above. All    interior points of $\gamma$   have  distance less than $d(p,q)$ to $p$ which implies that they lie   in $U$ as we require in the part (A) of Theorem \ref{connect-thm}. Let us show that all points of the geodesic are regular. 

Because this geodesic is a locally-Lipschitz curve of finite length, it has a regular point.  
As explained above, the set of  regular times  is   open  in $[0, d(p,q)]$ and so is a union of open intervals  of the form $[0, b),$ $  (a,b), $ or $(a, d(p,q)] $,
with $0 < a < b < d(p,q)$.    Observe though that  the intervals of the form $[0, b)$ and  $  (a,b) $ can not occur,  since  the  
endpoint $b$ is necessarily a regular point by Lemma \ref{lem:smooth}.     Thus, 
the set of regular points is either all of $[0, d(p,q)]$ or is $(0,d(p,q)]$.

It remains to show that only the first possibility occurs, that is, to  show that   the starting point $t=0$ is also regular.
For this purpose,  we  consider backward geodesics instead of forward geodesics.  (These
are geodesics for the Kropina structure  $-F = g / (-\omega)$.  
Clearly, $ \gamma( d(p,q)-t)$ is a   minimal arc-length parameterized  geodesic for the  backward distance).   All  arguments above  survive the  change from `forward' to `backward' geodesic.  The  change replaces   $t=0$ by $t= d(p,q)$   yielding  that   $t=0$ is regular.   Since now  all points of $\gamma$ are regular, it  is a  Kropina  geodesic of $F$.
 Theorem \ref{connect-thm} is proved.

\end{document}